\documentclass[reqno,12pt]{amsart}
\usepackage{amsmath, amssymb, amsthm, epsfig}
\usepackage{hyperref, latexsym}
\usepackage{times}
\usepackage{url}
\usepackage[mathscr]{euscript}
\usepackage{color}
\usepackage{setspace}
\usepackage[left=2cm, right=2cm, bottom=2cm, top=2cm]{geometry}
\usepackage{amsmath, amssymb, amsthm, epsfig}
\usepackage{hyperref, latexsym}
\usepackage{url}
\usepackage{sagetex}
\makeatletter
\@namedef{subjclassname@2020}{%
  \textup{2020} Mathematics Subject Classification}
\makeatother

\numberwithin{equation}{section}

\newenvironment{mysage}{\sagesilent}{\endsagesilent}
\def\today{\ifcase\month\or
  January\or February\or March\or April\or May\or June\or
  July\or August\or September\or October\or November\or December\fi
  \space\number\day, \number\year}

 \newtheorem{theorem}{Theorem}[section]
  
 \newtheorem{lemma}[theorem]{Lemma}
 \newtheorem{proposition}[theorem]{Proposition}
 \newtheorem{corollary}[theorem]{Corollary}
 \theoremstyle{definition}

 \newcommand{\mc}{\mathcal}

 \newcommand{\C}{\mathbb{C}}
 \newcommand{\R}{\mathbb{R}}
 \newcommand{\N}{\mathbb{N}}
 
 \newcommand{\Z}{\mathbb{Z}}

 \newcommand{\ds}{\text{\rm d}s}
 \newcommand{\dt}{\text{\rm d}t}
 \newcommand{\du}{\text{\rm d}u}

 \newcommand{\dx}{\text{\rm d}x}
 \newcommand{\dy}{\text{\rm d}y}

      \renewcommand{\d}{{\rm d}}

\newcommand{\newb}[1]{\textcolor{black}{#1}}

\begin{document}

\baselineskip=17pt

\title[Two-parameter quadratic sieve]{Optimality for the two-parameter quadratic sieve}
\author[Carneiro, Chirre, Helfgott and Mej\'{i}a-Cordero]{Emanuel Carneiro, Andr\'{e}s Chirre, \\ Harald Andr\'{e}s Helfgott and Juli\'{a}n Mej\'{i}a-Cordero}

\date{\today}
\address{
ICTP - The Abdus Salam International Centre for Theoretical Physics, Strada Costiera, 11, I - 34151, Trieste, Italy.}
\email{carneiro@ictp.it}

\address{Department of  Mathematical Sciences, Norwegian University of Science and Technology, NO-7491 Trondheim, Norway.}
\email{carlos.a.c.chavez@ntnu.no }

\address{Mathematisches Institut, Georg-August-Universit\"{a}t G\"{o}ttingen, Bunsenstrasse 3-5, 37073 G\"{o}ttingen, Germany.}
\email{harald.helfgott@gmail.com}

\address{Department of Mathematics, The Ohio State University, 231 W 18th Avenue, Columbus, Ohio, 43210, USA.}
\email{mejiacordero.2@osu.edu}

\allowdisplaybreaks

\begin{abstract}
We study the two-parameter quadratic sieve for an arbitrary smoothing function. We prove, under some very general assumptions, that the function considered by Barban and Vehov \cite{zbMATH03310988} and Graham \cite{zbMATH03593672} for this problem is optimal up to and including the second-order term. We determine that second-order term explicitly.
\end{abstract}
\maketitle

  \begin{mysage}
    cei=(lambda x,d : N((((RBF(10^d)*RBF(x)).upper()).ceil())/10^d))
    flo=(lambda x,d : N((((RBF(10^d)*RBF(x)).lower()).floor())/10^d))
    rou=(lambda x,d : N((((RBF(10^d)*RBF(x)).mid()).round())/10^d))
    import re
    def pf(x):
      return sage.misc.latex.LatexExpr(re.sub(r'(\d)\.?0+($|\ )',r'\1\2',latex(RR(x))).replace(r'\times',r'\cdot'))
  \end{mysage}
%  def pf(x):
%      return x.str(no_sci=2, skip_zeroes=True, truncate=True)
  \begin{mysage}
    import numpy as np
  \end{mysage}

\section{Introduction}

Consider functions $\rho: \Z_{>0} \to \R$ that satisfy $\rho(d)=1$ for $1\leq d\leq D_1$, and $\rho(d)=0$ for $d \geq D_2$, where $1 \leq D_1 < D_2$.
At the simplest level, what we may call a {\em quadratic sieve}
(with {\em sieve dimension} 1)
consists of a choice of $\rho$ for given $D_1$ and $D_2$,
with the injunction to choose $\rho$ so that
\begin{equation}\label{eq:sums}
  S_\rho = \sum_{1\leq n\leq N} \left(\sum_{d|n} \mu(d) \rho(d)\right)^2\end{equation}
  is small. This kind of sieve was introduced by Selberg.
  For general background and motivation, see \S \ref{subs:backmot}.
  
  \subsection{Selberg's and Barban-Vehov's choices of $\rho$}

  It is easy to see that
  \begin{equation}\label{eq:babieca}
S_\rho = M_\rho \cdot N + O(D_2^2),
  \end{equation}
 where\footnote{Here, as is usual, $[d_1,d_2]$ denotes the lowest common multiple of $d_1$ and $d_2$, while $(d_1,d_2)$ denotes their greatest common
    divisor.}
\begin{equation*}
M_\rho = \sum_{d_1,d_2} \frac{\mu(d_1) \mu(d_2)}{[d_1,d_2]} \rho(d_1) \rho(d_2).
\end{equation*}
One is then naturally led to the problem of minimizing
$M_\rho$ for given $D_1$ and $D_2$ (as we shall later remark, the case $D_2>\sqrt{N}$ is also important; indeed it is the case needed for the applications in \cite{zbMATH03590364} and \cite{HelfBook}, where the kind of sum that then arises is analogous but not identical to $M_\rho$).

\smallskip

For $D_1=1$, the choice $\rho=\rho^*$ such that $M_{\rho}$ is minimal was found
by Selberg in 1947 \cite{zbMATH03061658} (see also \cite[Chapter 7]{MR2647984}).
We then have
\begin{equation}\label{Selberg_sieve_expansion}
  M_{\rho^*} = \frac{1}{\sum_{d\leq D_2} \frac{\mu^2(d)}{\phi(d)}} = 
  \frac{1}{\log D_2} - \frac{c_0 + o(1)}{\log^2 D_2},
\end{equation}
where \[c_0 = \gamma + \sum_p \frac{\log p}{p (p-1)} =
1.33258227\dotsc.\]
Selberg's choice of $\rho^*(d)$ depends heavily on the
divisibility properties of $d$. For quite a few applications, it is
better to restrict the search to functions $\rho$ that are
scaled versions of a given continuous function $h:\R\to \mathbb{R}$, with $h(x)=0$ \,for $x\leq 0$ and $h(x)=1$ \,for $x \geq 1$. We can consider
$\rho = \rho_{D_1,D_2,h}: (0,\infty) \to \R$ given by
\begin{equation}\label{eq:poispidg}
\rho_{D_1,D_2,h}(t) := h\left(\frac{\log(D_2/t)}{\log (D_2/D_1)}\right).
\end{equation}
The reasons to define $\rho$ as a rescaling of a continuous $h$ 
are multiple: there is simplicity, which is particularly important
for an enveloping sieve (see \S \ref{Envel_Sieves}) or if $\rho$ appears as
a smooth cutoff for another, complementary sum; also -- though we will not focus on this
issue -- sieves of this kind can be made to yield results when
$D_2>\sqrt{N}$, a range that is outside the reach of more conventional sieves,
including Selberg's.

\smallskip

Sieves of this type -- that is, as in (\ref{eq:poispidg}), with $h$ continuous
-- were studied in depth from the late 60s to the early 80s
\cite{zbMATH03310988}, 
\cite{motohashi1974problem}, \cite{zbMATH03593672},
\cite{MR557127}, \cite{MR575649},
\cite{MR735437} and then seem to have lain half-dormant until their use by
Goldston and Y{\i}ld{\i}r{\i}m
\cite{goldston2002higher}, and much of what followed
 (\cite{MR2552109}, \cite{MR3373710}, \cite{MR3488739}, \cite{MR3783592}); see, however, the
application in \cite{MR1600529} and \cite{MR2256795},
and the use of $\rho_{D_1,D_2,h}(t)$ in the context of mollifiers
(\S \ref{subs:siedirich}). Recent work has centered
on their use and generalizations, rather than on what remained to be done
in the basic theory. Here, we will focus on some matters in the
basic theory that \newb{are} still not fully resolved,
\newb{remain} in an unsatisfactory state, or, at least, \newb{have} not been worked out
plainly and all in one place.

\smallskip

We will write $M(D_1,D_2;h)$ for $M_{\rho_{D_1,D_2,h}}$.
It has been long known (\cite{zbMATH03593672}) that Barban and Vehov's
choice of $h(x)$, namely, 
 \begin{equation}\label{eq:finh}
     h(x)=   h_0(x) = \begin{cases} 0 & \text{for $x<0$,}\\
          x &\text{for $0\leq x\leq 1$,}\\
          1 &\text{for $x>1$,}\end{cases}
          \end{equation}
 gives $M(1, D_2; h_0)$ with optimal main term (for $D_2\to \infty$),
 which is the same as the main term
 given by Selberg's sieve \eqref{Selberg_sieve_expansion} (Barban and Vehov had \newb{already} shown \cite{zbMATH03310988} that the main term
 in $M(1,D_2;h_0)$
 is bounded by a constant times the optimal main term).
 Thus, $h_0(x)$ was used in practice, although the lower-order terms of
 $M(1, D_2; h_0)$ do not seem to have been derived in the literature
 before \cite{HelfBook} and \cite{SZThesis}.
 
 \smallskip

 In the two-parameter case (that is, $D_1$ not necessarily equal to $1$),
 Barban and Vehov \cite{zbMATH03310988} proved that
 $M(D_1, D_2; h_0)  \ll (\log (D_2/D_1))^{-1}$, and Graham \cite{zbMATH03593672} went further by showing that 
\begin{equation}\label{eq:danzante}
  M(D_1, D_2; h_0) = \frac{1}{\log \frac{D_2}{D_1}}
  + O\left( \frac{1}{\log^2 \frac{D_2}{D_1}}\right).
\end{equation}
We can then ask ourselves here: (a) What is the second-order term in
$M(1, D_2; h_0)$ or, more generally, in $M(D_1, D_2; h_0)$?
(b) For which functions $h$ do we obtain the optimal main term?  (c) Out of those, for which do we obtain also the optimal second-order term? Part of the motivation for question (a) is that the error term in \eqref{eq:danzante} is rather large, and can be an obstacle to applications. For some applications, we actually need the second-order term to be explicit. Of course we are then especially interested
in question (c), since we want the second-order term to be as small as possible -- or rather, as far below
$0$ as possible, since it will actually turn out to be negative.

(There are, naturally, further questions one may ask oneself
once these are answered. See \S \ref{subs:futdir}.)

\subsection{Results}
Our first result refines \eqref{eq:danzante} by
describing the second-order term for the choice $h_0$.
See \S \ref{Prev_lit} for a discussion of the alternative approach in
\cite{HelfBook} and \cite{SZThesis}.
\begin{theorem}\label{Thm2} 
Let $1 \leq D_1 < D_2$ and let $h_0$ be given by \eqref{eq:finh}. 
\begin{itemize}
\item[(i)] If $D_1 =1$, 
\begin{equation}\label{20200304_12:53}
 M(1, D_2; h_0) = \frac{1}{\log D_2} - \frac{\kappa}{\log^2 D_2} + O\left(\frac{e^{- C \sqrt{\log D_2}}}{\log^2 D_2} \right)
\end{equation}
for some $C>0$.

\smallskip

\item[(ii)] In the general case $1 \leq D_1 < D_2$\,,
 \begin{equation}\label{20200227_17:02}
M(D_1, D_2; h_0) = \frac{1}{\log \frac{D_2}{D_1}} - \frac{2\kappa}{\log^2 \frac{D_2}{D_1}} + 
O\left(\frac{e^{-C \sqrt{\log D_2/D_1}} +
  e^{-C \sqrt{\log D_1}}}{\log^2 \frac{D_2}{D_1}}\right)
\end{equation}
for some $C>0$. Here
\begin{mysage}
  kappa = 0.607314
\end{mysage}
\begin{equation}\label{eq:andian}
  \kappa = \frac{1}{2\pi} \int_{-\infty}^\infty \left(t^2 -
    \frac{H(t)}{|\zeta(1 + i t)|^2}\right) \frac{\dt}{t^4} = \sage{pf(kappa)}\dotsc,
\end{equation}
with
      \[H(t) = \prod_p \left(1 - \frac{2}{p^2}\frac{1 - \cos(t\log p)}{\big(1 - 2 p^{-1} (\cos t \log p) + p^{-2}\big)}\right)
      .\]
      \end{itemize}
\end{theorem}
We determine the numerical value of the constant $\kappa$ in \eqref{eq:andian}
rigorously, by means of interval arithmetic, and its variant,
ball arithmetic.\footnote{The packages
  used were ARB \cite{ARB}, for ball arithmetic, and MPFI \cite{ReRo02}, for interval arithmetic. Many smaller computations are included in the TeX source,
  via SageTeX; they take a total of a couple of seconds on modern equipment.
  All other computations are included in a Jupyter/Sagemath worksheet, to
  be found in the arXiv submission. Their total running time is
somewhere between a long coffee break and a tea hour.}
 See Section \ref{sec:beautlaund}.
 
\smallskip

We now move towards understanding the optimality of the function $h_0$.
Write $|f|_2$ for the $L^2(\R)$-norm of a function $f$ defined almost everywhere on $\R$, and
$V(f)$ for the infimum of the total variation $|g|_{\rm{TV}}$ over
all functions $g:\R\to \R$ that are equal to $f$ almost everywhere.\footnote{In fact, we could say ``minimum'': if $f$ is of bounded variation,
  there exists a representative $g$ that actually attains the minimum of the total variation in the equivalence class. (See \cite[Lemma 3.3]{zbMATH05120644} or
 \cite[Thms.~3.27 and 3.28]{zbMATH01448982} for recent
  references; the statement is surely older.) We use this fact for simplicity in some of our arguments below, but it is not crucial; we could just as well work with minimizing sequences.} Our next result is related to some general results in the literature (see \S \ref{Prev_lit} below), though they usually consider smooth and compactly supported test functions and focus on the main asymptotic term.
We will be more specific, in that we will give a bound on the error term while working under much weaker regularity assumptions on the smoothing function. We will provide a concise but self-contained proof, in part because
parts of it are also used in the proof of Theorem \ref{Thm2}.

\begin{theorem}\label{prop4}
  Let $1 \leq D_1 < D_2$. For $i \in \{1,2\}$, let $h_i:\R\to \mathbb{R}$ be an absolutely continuous function with $h_i(x)=0$ \,for $x\leq 0$ and $h_i(x)=h_i(1)$ \,for $x \geq 1$, and such that $V(h_i') < \infty$. Then, for
  $\rho_{i} = \rho_{D_1,D_2,h_i}$ given by
  \eqref{eq:poispidg},
\begin{equation}\label{20200612_14:21}
  \sum_{d_1,d_2} \frac{\mu(d_1) \mu(d_2)}{[d_1,d_2]} \rho_{1}(d_1)
  \rho_{2}(d_2) =  \frac{\int_{-\infty}^{\infty} h_1'(x)\,h_2'(x)\,\dx}{\log \frac{D_2}{D_1}} + O\left( \frac{V(h_1')V(h_2')}{\log^2 \frac{D_2}{D_1}}\right).
\end{equation}
\end{theorem}
\noindent {\em Remark}: Note that we are not assuming here that
$h_i(1)$ is $1$.

\smallskip

From Theorem \ref{prop4} we easily get the following result.

\begin{corollary}\label{Prop1}
Let $1 \leq D_1 < D_2$ and let $h_0$ be given by \eqref{eq:finh}. 
\begin{enumerate}
\item[(i)] Let $h:\R\to \mathbb{R}$ be an absolutely continuous function, with $h(x)=0$ \,for $x\leq 0$ and $h(x)=1$ \,for $x \geq 1$, and such that $V(h') < \infty$. Then 
 \begin{equation}\label{20200211_12:34}
 M(D_1, D_2; h) = \frac{|h'|_2^2}{\log \frac{D_2}{D_1}} + O\left( \frac{V(h')^2}{\log^2 \frac{D_2}{D_1}} \right).
 \end{equation} 
 \item[(ii)] Let $h_1:\R\to \mathbb{R}$ be an absolutely continuous function, with $h_1(x)=0$ \,for $x\leq 0$ and $h_1(x)= 0$ \,for $x \geq 1$, and such that $V(h_1') < \infty$. Let $g:[1,\infty) \to [1,\infty)$ be a given function and set 
 \begin{equation}\label{20200212_17:20}
 h(x) = h_0(x) + \frac{1}{g\big(\log \frac{D_2}{D_1}\big)} \,h_1(x).
 \end{equation}
 Then
 \begin{equation}\label{20200612_09:12}
 M(D_1, D_2; h) = M(D_1, D_2; h_0) + \frac{|h_1'|_2^2}{g\big(\log \frac{D_2}{D_1}\big)^2\,\log \frac{D_2}{D_1}} + O\left( \frac{V(h_1') + V(h_1')^2}{g\big(\log \frac{D_2}{D_1}\big)\,\log^2 \frac{D_2}{D_1}} \right).
 \end{equation}
 \end{enumerate}
\end{corollary}

In the situation of Corollary \ref{Prop1} (i), by Cauchy-Schwarz, $|h'|_2 \geq 1$, with equality if and only if $h = h_0$, for $h_0$ given by \eqref{eq:finh}:
\[|h'|_2^2 = \int_{0}^1 |h'(x)|^2 \d x \geq
\left|\int_0^1 h'(x) \d x\right|^2 = |h(1)-h(0)|^2 = 1.\]

Corollary \ref{Prop1} (ii) says a little more about the uniqueness of the function $h_0$ as the optimizer, in that one cannot get second-order gains (which might be hiding in the error term in \eqref{20200211_12:34}) by means of small perturbations as in \eqref{20200212_17:20}. In fact,
equation (\ref{20200612_09:12}) implies that
$M(D_1,D_2;h)>M(D_1,D_2;h_0)$ unless 
\begin{equation}\label{eq:dederemos}
  g(t) \gg \frac{|h_1'|_2^2}{V(h_1')+V(h_1')^2}\cdot t.\end{equation}
If \eqref{eq:dederemos} holds, then
\[M(D_1,D_2;h) = M(D_1,D_2;h_0) + O\left(\frac{C}{ \log^3 \frac{D_2}{D_1}}\right).\]
with $C = (V(h_1')+V(h_1')^2)^2/|h_1'|_2^2$.

%If we consider, for instance, a situation where $g(x) \to \infty$ as $x \to \in%fty$, one readily sees that the error term in \eqref{20200612_09:12} is already smaller than the second-order term in \eqref{20200227_17:02}.

\smallskip

Let us summarize our results. The sieve sum $M(D_1,D_2;h)$
for $h = h_0$ is as stated in (\ref{20200304_12:53}) and (\ref{20200227_17:02});
in particular, there is a relatively large negative second order term
and a small error term. Corollary \ref{Prop1} asserts that
the value of $M(D_1,D_2;h)$ is minimal for
$h = h_0$, at least in so far as the main term and the second-order term
are concerned.

As a side remark, we should note that, for $h=h_0$, the error term in (\ref{eq:babieca}) is in fact $O_\rho(D_2^2/\log^2 D_2)$ and not
just $O_\rho(D_2^2)$. It should be clear to the reader that this sharper bound
holds for any continuous $h(x)$ such that $h(x) = O(x)$ as $x\to 0^+$ (the same is true for Selberg's $\rho=\rho^*$; see, e.g.,
\cite[\S 7.11]{MR2647984}).

\subsection{Relation to the previous literature}\label{Prev_lit}

In the most classical case $D_1=1$, the main term $|h'|_2^2/\log D_2$ in
Corollary \ref{Prop1} (i)
was surely known.
A more general asymptotic
appears in \cite[Lemma 4.1]{MR3373710} (with an error term of size $o(1)$ times the main term), accompanied by a mention
that ``such asymptotics are standard in the literature'', but
earlier appearances seem hard to pin down. Given that the main term is proportional to $|h'|_2^2$, determining when it is optimal reduces to a simple application of Cauchy-Schwarz. The two-parameter problem, with $1 \leq D_1 < D_2$, was considered in \cite{zbMATH03310988}, \cite{zbMATH03593672} for $h = h_0$, and in \cite{MR557127} for $h(x) = x^k$ for $0\leq x \leq 1$, $k \in \N \cup \{0\}$, in a slightly different setup than our Corollary \ref{Prop1} (i). 

\smallskip

A novelty here relative to the older literature is that we can compute lower-order terms, both in the one-parameter ($D_1=1$) and two-parameter ($D_2> D_1
\geq 1$)
cases, and show that Barban and Vehov's choice for $h_0$
(namely, $h_0$ as in
\eqref{eq:finh}) remains optimal even when we take them under consideration. One of the difficulties involved in proving Theorem \ref{prop4} and its resulting Corollary \ref{Prop1} is to show
that the error term is small whenever $\log (D_2/D_1)$ is large.
For instance, with some work, one can get the right main term in Corollary \ref{Prop1} (i) for
$h(x)$ a polynomial from the work of Jutila \cite[Theorem 1]{MR557127}, but the error term is then not of the desired size, or even smaller than the main term, for $D_1$ and $D_2$ completely
arbitrary.
One difference is that we work with a double contour integral and then extract a single contour integral as the main term, whereas
\cite{MR557127} and \cite{MR1985667} use a single contour integral to estimate a sum that appears within another sum.
The same double contour integral studied here
appears in \cite{MR3373710} and \cite{MR3488739}, but the procedure followed there is somewhat different. A double contour integral also appears in the
study of the unsmoothed sum in \cite{zbMATH05030051}.

\smallskip

The sum $S$ in \eqref{eq:sums} received a fully explicit estimate in \cite[Chapter 7]{HelfBook} in the ranges
$D_2\ggg \sqrt{N}$ (one-parameter case) and $D_2\geq D_1\ggg \sqrt{N}$ (two-parameter case), which were considered by \cite{zbMATH03310988}
and \cite{zbMATH03593672}, but which we do not study here. The approach
in \cite[Ch.~7]{HelfBook} (at least in its version from 2017--2019)
is rather more
real-analytic than in the present paper.
The sum $M(D_1, D_2; h)$ was then estimated in a related way in Sebasti\'an Z\'u\~niga Alterman's thesis \cite{SZThesis}, thus making
it possible to estimate $S$ for $D_2\lll \sqrt{N}$.
In particular, \cite{SZThesis} works out the second order term in
\eqref{20200304_12:53} and \eqref{20200227_17:02} for parameters
$D_1$, $D_2$ in wide ranges.
%, that is, in the one-parameter case $D_1=1$.

The point of \cite[Chapter 7]{HelfBook} and \cite{SZThesis} is to
give good, fully explicit estimates, rather than to prove optimality. All the same, \cite{SZThesis} succeeds in computing the first
three digits of the constant $\kappa$ in \eqref{eq:andian}, proving their correctness. The
values of $\kappa$ determined by our method and that of \cite{SZThesis} naturally coincide. The recent preprint \cite{SZnew} gives
the value $\kappa = 0.60731\dotsc$ in the case $D_1=1$,
again agreeing with our value for $\kappa$. We compute one more digit,
in part to demonstrate that our method can be pushed further with ease,
being essentially self-contained.

%\newb{in \eqref{eq:andian}, which brings the sixth digit.} \newr{In both \cite{SZThesis} and \cite{SZnew} the constant $\kappa$ appears in connection with the one-parameter situation $D_1=1$.}

%Incidentally, the application in \cite{HelfBook}
%shows a context in which it is natural to require that the weight $\rho$ be
%of the form 
%$\rho = \rho_{D_1,D_2,h}$ as in \eqref{eq:poispidg}, rather than, say,
%Selberg's optimal $\rho = \rho^*$. 
%In \cite{HelfBook}, the expression $S$ in (\ref{eq:sums}) arises from
%Vaughan's identity, followed by Cauchy-Schwarz. To be precise, what is used
%is a smoothed form of Vaughan's identity, with a continuous, monotonic $\rho$
%(to be chosen at will) replacing a sharp cut-off. It would not do to use
%$\rho^*$ in such a context, as it is not monotonic. To put matters otherwise: s%ums such as $S$ can appear naturally, even if we are not minded to sieve.

\subsection{Notation} For the rigorous numerical evaluation parts, for $\beta >0$, we say that $\alpha = O^*(\beta)$ when $\alpha \in [-\beta, \beta]$. Other symbols such as $\ll$\,, $O(\cdot)$ or $o(\cdot)$ are used in the standard way.

\section{Proof of Theorem \ref{prop4}}

Throughout the proof we let $L = \log (D_2/D_1)$ be our desired scale and let $M$ be the sum on the left-hand side of \eqref{20200612_14:21}. Without loss of generality, we assume that we are choosing representatives $h_i'$ of bounded variation such that $|h_i'|_{TV} = V(h_i') < \infty$ for
  $i=1,2$.

\subsection{Mellin transform and the integral formulation} For $\Re(s) >0$, let $F_i$ be the Mellin transform of $\rho_i$ defined by 
$$F_i(s) = \int_0^{\infty} x^{s-1} \rho_i(x)\, \dx.$$
Integration by parts yields
\begin{equation}\label{20191018_22:34}
F_i(s) = -\frac{1}{s} \int_0^{\infty} x^{s} \rho_i'(x)\, \dx = -\frac{1}{s} \int_{D_1}^{D_2} x^{s} \rho_i'(x)\, \dx,
\end{equation}
from which we see that $F_i$ can be extended to a meromorphic function over $\C$ with a simple pole at $s=0$ with residue $h_i(1)$. Moreover, a further application of integration by parts (\newb{and here it is important that $h_i'$ is a function and not merely a measure, which is the reason we assume $h_i$ to be absolutely continuous from the start}) yields
\begin{align}\label{20200227_17:22}
F_i(s) &=  -\frac{1}{s} \int_0^{\infty} x^{s} \rho_i'(x)\, \dx  = \frac{D_2^s}{s} \int_{-\infty}^{\infty} e^{-tLs}\, h_i'(t) \,\dt = \frac{D_2^s}{Ls^2}\int_{-\infty}^{\infty} e^{-tLs}\, \d h_i'(t),
\end{align}
which then implies
\begin{align}\label{growth_est}
  |F_i(s)| \leq \frac{D_2^{\sigma}}{L|s|^2} \max\{1, e^{-\sigma L}\} \int_{-\infty}^{\infty} |\d h_i'(t)| = \frac{\max\{D_1^{\sigma}, D_2^{\sigma}\}}{L|s|^2} \,
  V(h_i'),
\end{align}
for all $s \in \C$, where $\sigma = \Re(s)$. Mellin inversion then yields
\begin{equation}\label{Mellin_inversion}
\rho_i(x) =   \frac{1}{2\pi i } \int_{(\sigma) } x^{-s} F_i(s) \, \ds
\end{equation}
for $x >0$, where $\sigma >0$ and the integration runs over the vertical line $(\sigma):= \{s \in \C\,:\, \Re(s) = \sigma\}$.

\smallskip

Using \newb{\eqref{growth_est} and \eqref{Mellin_inversion}} (to justify the use of Fubini's theorem below) we can write
%rewrite \newb{\eqref{20200612_14:21}} as 
\begin{align*}
  M &= \sum_{d_1,d_2} \frac{\mu(d_1) \mu(d_2)}{[d_1,d_2]} \rho_1(d_1) \rho_2(d_2)\\
  &= \frac{1}{(2\pi i )^2} \int_{(\sigma_2)} \! \int_{(\sigma_1)} \!\! F_1(s_1)
  F_2(s_2) \left(\sum_{d_1,d_2} \frac{\mu(d_1) \mu(d_2) d_1^{-s_1} d_2^{-s_2}}{[d_1,d_2]}\right) \! \ds_1\,\ds_2.
\end{align*}
with $\sigma_1, \sigma_2 >0$. A routine computation yields
\begin{equation*}
\sum_{d_1,d_2} \frac{\mu(d_1) \mu(d_2) d_1^{-s_1} d_2^{-s_2}}{[d_1,d_2]} = \frac{\zeta(1 + s_1 + s_2)}{\zeta(1+s_1) \zeta(1 + s_2)} \,G(s_1, s_2),
\end{equation*}
with $\zeta$ being the Riemann zeta-function and 
\begin{equation}\label{20200227_17:13}
G(s_1, s_2) = \prod_p \left(1 + \frac{p^{1+s_1} + p^{1+s_2} - p^{1+s_1+s_2} - p}{p^{1 + s_1 + s_2} (p^{1 + s_1} - 1) (p^{1 + s_2} - 1)}\right).
\end{equation}
Note that $G$ is uniformly bounded in the region
\begin{equation}\label{eq:gregion}
  \mathscr{R} = \big\{(s_1, s_2) \in \C^2 \, :\,  \Re(s_1) \geq  -\frac{1}{5} \ {\rm and}\ \Re(s_2) \geq  -\frac{1}{5}\big\}.
\end{equation}
Our task then becomes to study the double integral 
\begin{align*}
M = \frac{1}{(2\pi i )^2} \int_{(\sigma_2)} \! \int_{(\sigma_1)} \!\!  \frac{\zeta(1 + s_1 + s_2) F_1(s_1) F_2(s_2)}{\zeta(1+s_1) \zeta(1 + s_2)} \, G(s_1, s_2) \,\ds_1\,\ds_2.
\end{align*}
We will proceed by applying the residue theorem.

\subsection{Shifting the contours of integration} \label{Shift_contour}We make use of the classical zero-free region of  $\zeta$,
and a standard bound on $1/\zeta(s)$ therein, as described in, say
\cite[Thms.~6.6--6.7]{zbMATH06069874} (naturally, somewhat stronger results can be obtained by means of
a Vinogradov-style zero-free region). There exists a constant $c_0 >0$ such that 
\begin{equation}\label{20191018_20:51}
\frac{1}{|\zeta(s)|} \ll \log(|t| +2)
\end{equation}
for $s = \sigma + it$ uniformly in the region $\sigma \geq 1  -c_0 (\log (|t|+2))^{-1}$. Let $\mathcal{C}$ be the contour given by $\mathcal{C} = \{s = \sigma + it \, : \, \sigma = -c_0 (\log (|t|+2))^{-1}\}$. Note that $1 + \mathcal{C}= \{1 + s\, :\, s \in \mathcal{C}\}$ falls in the zero-free region for $\zeta$. Fix $\sigma_2 >0$ small, but still such that $\sigma_2 > -\Re(s)$ for
all $s \in \mathcal{C}$. Then, when we move the contour in the inner integral
below we pick up no poles and get
\begin{align}
M & =  \frac{1}{(2\pi i )^2} \int_{(\sigma_2)} \left( \int_{(\sigma_1)} \frac{\zeta(1 + s_1 + s_2) \,F_1(s_1)\,G(s_1, s_2)}{\zeta(1+s_1)} \, \ds_1\right) \frac{F_2(s_2)}{\zeta(1 + s_2)} \,\ds_2 \nonumber \\
& = \frac{1}{(2\pi i )^2} \int_{(\sigma_2)} \left( \int_{\mathcal{C}}   \frac{\zeta(1 + s_1 + s_2) \,F_1(s_1)\,G(s_1, s_2)}{\zeta(1+s_1)} \, \ds_1\right) \frac{F_2(s_2)}{\zeta(1 + s_2)} \,\ds_2 \nonumber \\
& = \frac{1}{(2\pi i )^2} \int_{\mathcal{C}} \left( \int_{(\sigma_2)}   \frac{\zeta(1 + s_1 + s_2) \, F_2(s_2)\,G(s_1, s_2)}{\zeta(1+s_2)} \, \ds_2\right) \frac{F_1(s_1)}{\zeta(1 + s_1)} \,\ds_1, \label{20191019_20:31}
\end{align}
where the use of Fubini's theorem in the last passage is justified from the decay estimates \eqref{growth_est} and \eqref{20191018_20:51}. Now, for each fixed $s_1$ in \eqref{20191019_20:31}, we shift the contour in the inner integral to $\mathcal{C}$ picking up a simple pole when $s_2 = -s_1$. Hence
\begin{align}\label{20191018_20:42}
\begin{split}
  M &=
  \frac{1}{2\pi i } \int_{\mathcal{C}}\frac{G(s_1, -s_1) F_2(-s_1)F_1(s_1)}{\zeta(1 - s_1)\zeta(1 + s_1)} \,\ds_1\\
  + \frac{1}{(2\pi i )^2} & \int_{\mathcal{C}} \left( \int_{\mathcal{C}}   \frac{\zeta(1 + s_1 + s_2) \,F_2(s_2)\,G(s_1, s_2)}{\zeta(1+s_2)} \, \ds_2  \right) \frac{F_1(s_1)}{\zeta(1 + s_1)} \,\ds_1 .
\end{split}
\end{align}

\subsection{Error term: double integral}\label{Error_Double_Integral}
We \newb{now show} that
the double integral in \eqref{20191018_20:42}
is bounded by the error term in \eqref{20200612_14:21}.
We use the fact that $G$ is
uniformly bounded in the region $\mathscr{R}$ defined in \eqref{eq:gregion},
together with the decay estimates
\eqref{growth_est} and \eqref{20191018_20:51}. Recall that if
$s \in \mathcal{C}$ then $|s| \geq c$, where $c$ is an absolute constant,
and so \eqref{growth_est} is suitable in the whole range, that is,
even when $\Im (s)$ is small.
It is also convenient to use an estimate of the type
\begin{align*}
  |\zeta(1 + s_1 + s_2)|  & \ll \max\{|s_1+s_2|^{-1}, |s_1+s_2|^{1/4}\} \\
  &\ll \max\big\{\log (2 + |s_1| + |s_2|), |s_1+s_2|^{1/4}\big\}\\
& \ll \big( \log (2 + |s_1|) + |s_1|^{1/4}\big) \big( \log (2 + |s_2|) + |s_2|^{1/4}\big),
\end{align*}
valid \newb{when $s_1, s_2 \in \mc{C}$}. The first inequality comes from the simple pole of $\zeta$ at
$1$ and basic convexity estimates in the critical strip
(see, e.g., \cite[(5.1.5)]{MR882550}; of course one can also
use stronger results, such as \cite[Thm.~5.12]{MR882550}).
The second inequality follows from the definition of the contour
$\mathcal{C}$ since
\begin{align*}
  |s_1 + s_2| &\geq |\Re(s_1 + s_2)| = c_0\left| \frac{1}{\log (2 + |\Im(s_1)|)} +
  \frac{1}{\log (2 + |\Im(s_2)|)}\right|\\ &\gg \frac{1}{\log (2 + |\Im(s_1)| +
    |\Im(s_2)|)}.
\end{align*}
With these estimates at hand, we obtain
\begin{align*}
  &\left|\int_{\mathcal{C}}  \int_{\mathcal{C}}   \frac{\zeta(1 + s_1 + s_2)
    \,F_2(s_2)\,F_1(s_1)\, G(s_1, s_2)}{\zeta(1+s_2)\, \zeta(1 + s_1)} \, \ds_2
  \,\ds_1\right| \\
  \!\!\!\!\! \ll  \! \frac{V(h_1')V(h_2')}{L^2} &\!\left| \int_{\mathcal{C}} \!\int_{\mathcal{C}}
  \! \frac{
    \prod_{j=1}^2 \big(\!\log (2 + |s_j|) \!+\!|s_j|^{\frac14}\big)
    \!\log (2 + |s_j|) \cdot  
    D_1^{\Re(s_1 + s_2)} }{|s_1|^2\,|s_2|^2} \, \ds_2 \,\ds_1\right|\\
  \!\!\!\!\! \ll \frac{V(h_1')V(h_2')}{L^2} &\left| \int_{\mathcal{C}} \frac{\big(\log (2 + |s|)
    + |s|^{\frac14}\big)\,\log (2 + |s|) \,D_1^{\Re(s)}}{|s|^2}\,\ds\right|^2
  \\\ll \frac{V(h_1')V(h_2')}{L^2} &e^{- C\sqrt{\log D_1}}.
\end{align*}
for some $C>0$.
The last inequality follows by breaking the integral in two, at a height
$|t| \sim e^{\sqrt{\log D_1}}$.

\subsection{Main term}
Shifting the contour $\mathcal{C}$ in
the simple integral in \eqref{20191018_20:42}
back to the imaginary axis, we arrive at 
\begin{equation}\label{20191018_23:10}
  M = \frac{1}{2\pi  } \int_{-\infty}^{\infty}\frac{G(it, -it)
    F_1(it) F_2(-it)}{|\zeta(1 +
    it)|^2} \,\dt + O\left( \frac{V(h_1')V(h_2')}{L^2}\,e^{-C \sqrt{\log D_1}}\right).
\end{equation}
We \newb{now investigate the main term arising from} the integral in
\eqref{20191018_23:10}.
Let $g_i:\R \to \R$ be defined by 
$$g_i(x) := \rho_i(e^x) = h_i\left(\frac{\log D_2 - x}{L}\right).$$
Identity \eqref{20191018_22:34}, which is initially valid for $\Re(s) >0$, can
be rewritten as
\begin{align}\label{20191018_22:46}
-s F_i(s) = \int_{\log D_1}^{\log D_2} e^{ys}\, g_i'(y)\,\dy.
\end{align}
Now observe that both sides of \eqref{20191018_22:46} are entire functions;
hence, the identity is valid in the whole $\C$ by analytic continuation. In
particular, for $s = -2 \pi i t$ we obtain
\begin{align*}
  2\pi i t \
  F_i(-2\pi i t) = \int_{\log D_1}^{\log D_2} e^{-2 \pi i t y}\, g_i'(y)\,\dy,
\end{align*}
and so the function $t \mapsto 2\pi i t \ F_i(-2\pi i t)$ is the Fourier
transform of $g_i'$. By Plancherel's theorem we have
\begin{align}\label{20191018_23:00}
  \frac{1}{L} \int_{-\infty}^{\infty} h_1'(x)\,h_2'(x)  \,\dx = \int_{-\infty}^{\infty}
  g_1'(x)\,g_2'(x)  \,\dx =  \frac{1}{2\pi} \int_{-\infty}^{\infty} t^2\,
  F_1(it)F_2(-it)\,\dt.
\end{align}

\smallskip

We propose that \eqref{20191018_23:00} is the main
term we seek. We see from \eqref{20191018_23:10} that it remains to prove
that 
\begin{equation}\label{20191018_23:23}
  \int_{-\infty}^{\infty}\left( \frac{G(it, -it)}{|\zeta(1 + it)|^2} - t^2\right)\!
 F_1(it)F_2(-it)\,\dt = O\left( \frac{V(h_1')V(h_2')}{L^2}\right).
\end{equation}
The function $W(t) = \frac{G(it, -it)}{\zeta(1 + it)\zeta(1-it)}$ is
an even function of $t$ that is analytic in a region containing the real line.
Since $G(0,0) =1$ we have $W(t) - t^2 = O(t^4)$ as $t \to 0$ (see
\eqref{19_14_19_02} below for an explicit estimate). Combining this estimate
with the bound from \eqref{growth_est} we obtain that the
segment $t\in [-1,1]$ (say) 
contributes $O(V(h_1')V(h_2')/L^2)$ to 
the integral on the left side of \eqref{20191018_23:23}. Since $G(it,-it)$ is uniformly bounded in a neighborhood of the real line,
we also have the estimate $|W(t)| \ll (\log |t|)^2 $ for large $t$.
Using the decay estimate \eqref{growth_est} again, we conclude that
\eqref{20191018_23:23} holds. We are thus done with the proof of Theorem
\ref{prop4}.

\section{Proof of Corollary \ref{Prop1}}
Part (i) follows directly from Theorem \ref{prop4},
so we focus on part (ii). Let $L = \log (D_2/D_1)$, $\rho_{0} = \rho_{D_1,D_2,h_0}$ and $\rho_{1} = \rho_{D_1,D_2,h_1}$. We start \newb{by} noticing that, by linearity,
\begin{align}\label{20200227_16:31}
\begin{split}
& M(D_1, D_2; h)  = M(D_1, D_2; h_0) \\
&  \ \ \ \ \  + \frac{2}{g(L)} \sum_{d_1,d_2} \frac{\mu(d_1)
  \mu(d_2)}{[d_1,d_2]} \rho_{0}(d_1) \rho_{1}(d_2) + \frac{1}{g(L)^2}
\sum_{d_1,d_2} \frac{\mu(d_1) \mu(d_2)}{[d_1,d_2]} \rho_{1}(d_1) \rho_{1}(d_2). 
 \end{split} 
\end{align}
For our particular choice of $h_0$, we have
\begin{equation*}
  \int_{-\infty}^{\infty} h_0'(x)\,h_1'(x)\,\dx = \int_0^1h_1'(x)\,\dx = h_1(1) -
  h_1(0) = 0\,,
\end{equation*} 
and we may use Theorem \ref{prop4} to arrive at 
\begin{align}\label{20200227_16:32}
  \frac{2}{g(L)} \sum_{d_1,d_2} \frac{\mu(d_1) \mu(d_2)}{[d_1,d_2]} \rho_{0}(d_1)
  \rho_{1}(d_2) = O\left( \frac{V(h_1')}{g(L) L^{2}}\right).
\end{align}
Another application of Theorem \ref{prop4} yields
\begin{align}\label{20200227_16:33}
  \frac{1}{g(L)^2}\sum_{d_1,d_2} \frac{\mu(d_1) \mu(d_2)}{[d_1,d_2]} \rho_{1}(d_1)
  \rho_{1}(d_2) = \frac{\int_{-\infty}^{\infty} h_1'(x)^2\,\dx}{g(L)^2\,L} +
  O\left(
  \frac{V(h_1')^2}{g(L)^2\,L^{2}}\right).
\end{align}
Combining \eqref{20200227_16:31}, \eqref{20200227_16:32} and \eqref{20200227_16:33}, we conclude that Corollary \ref{Prop1} holds.

%\smallskip
%\noindent
%{\em Remark.}
\section{The second-order term}\label{sec:beautlaund}

In this section we prove Theorem \ref{Thm2}. We keep denoting
$L = \log (D_2/D_1)$.
We shall use some passages of the proof of Theorem \ref{prop4} for
$h_1=h_2 = h_0$, where $h_0$ is given by \eqref{eq:finh}.

\subsection{Two-parameter case: $1 \leq D_1 < D_2$} \label{Two_par_sec_term}

In this case, recall that in
\S \ref{Error_Double_Integral} we already showed that the contribution of the
double integral in \eqref{20191018_20:42} is incorporated in the proposed error
term in \eqref{20200227_17:02}. Therefore, the second-order term comes from the
evaluation proposed in \eqref{20191018_23:10}--\eqref{20191018_23:23}, namely
\begin{equation}\label{20200304_12:50}
\Sigma = \frac{1}{2\pi} \int_{-\infty}^\infty
\left(\frac{H(t)}{|\zeta(1 + i t)|^2} - t^2\right) |F(i t)|^2 \,\dt,
\end{equation}
where 
\begin{equation}\label{eq:Htdef}
H(t) = G(it,-it) =
\prod_p \left(1 + \frac{p^{it} + p^{-it}-2 }{\big(p^{1+it}-1\big) \big(p^{1- i t}
  - 1\big)}\right)
  \end{equation}
according to \eqref{20200227_17:13}, and $F(s)$ is the Mellin transform of
$\rho_{D_1,D_2,h_0}$ given by \eqref{eq:poispidg} with $h_0$ defined in \eqref{eq:finh}.
In this case, the explicit computation \eqref{20200227_17:22} yields
\begin{equation}\label{20200304_12:55}
F(s) = \frac{D_2^s - D_1^s}{L s^2}.
\end{equation}
Hence
\begin{align}
  \Sigma &= \frac{1}{2\pi L^2} \int_{-\infty}^\infty
  \left(\frac{H(t)}{|\zeta(1+it)|^2}
  - t^2\right) \frac{(D_2^{it} - D_1^{it})(D_2^{-it} - D_1^{-it})}{t^4}\,\dt
  \nonumber\\
  & = - \frac{2\kappa}{L^2} -\frac{1}{\pi L^2} \, \Re \left(\int_{-\infty}^\infty
  \left(\frac{H(t)}{|\zeta(1+it)|^2}
- t^2\right) \frac{(D_2/D_1)^{i t}}{t^4} \,\dt\right), \label{20200227_17:35}
\end{align}
where 
\begin{equation}\label{20200304_15:06}
\kappa = \frac{1}{2\pi}\int_{-\infty}^\infty 
\left(t^2 - \frac{H(t)}{|\zeta(1+it)|^2}\right) \frac{\dt}{t^4}.
\end{equation}
We shall first show that the oscillatory integral in \eqref{20200227_17:35}
contributes to the error term in \eqref{20200227_17:02} and, after that, only
the numerical computation of $\kappa$ will be missing.

\smallskip

It is easy to see that $H(s) = G(i s, - i s)$ is analytic and bounded on
any strip of the form $|\Im(s)|\leq c$, $0<c<1$. Hence we can shift
our contour to $-i \mathcal{C}$,
with $\mathcal{C}$ as in \S \ref{Shift_contour}, to get
\[\begin{aligned}\int_{-\infty}^\infty
  \left(\frac{H(t)}{|\zeta(1+it)|^2}
  - t^2\right) \frac{(D_2/D_1)^{i t}}{t^4} \,\dt &=
  \int_{-i \mathcal{C}} \left(\frac{H(t)}{|\zeta(1+it)|^2}
  - t^2\right) \frac{(D_2/D_1)^{i t}}{t^4}
  \\ &= O\left(e^{-C \sqrt{\log \frac{D_2}{D_1}}}\right)\end{aligned}\]
  for some $C>0$, where we obtain the bound on the last line splitting
  the contour $\mathcal{C}$ as before. Hence
  \[\Sigma = - \frac{2 \kappa}{L^2} + O\left(
  \frac{e^{-C \sqrt{\log \frac{D_2}{D_1}}}}{L^2}\right).\]
  
\subsection{The case $D_1 =1$} In the one-parameter case, i.e. $D_1 = 1$, the previous discussion in \S \ref{Two_par_sec_term} continues to hold for the analysis of the term $\Sigma$ defined in \eqref{20200304_12:50}, but we must take a closer look at the double integral appearing in \eqref{20191018_20:42}, since the reasoning of \S \ref{Error_Double_Integral} is not sufficient anymore in order to reach the proposed error term in \eqref{20200304_12:53}. 

\smallskip

We want to look at the following term from \eqref{20191018_20:42}
\begin{equation*}
\Sigma_2 := \frac{1}{(2\pi i )^2}  \int_{\mathcal{C}}  \int_{\mathcal{C}}   \frac{\zeta(1 + s_1 + s_2) \,F(s_1)\,F(s_2)\,G(s_1, s_2)}{\zeta(1 + s_1) \zeta(1+s_2)} \, \ds_1 \,\ds_2\,,
\end{equation*}
where the contour $\mathcal{C}$ is defined in \S \ref{Shift_contour}. From \eqref{20200304_12:55}, in this one-parameter case, we have 
\begin{equation*}
F(s) = \frac{D_2^s - 1}{L s^2}.
\end{equation*}
Hence we may simply multiply out to get
\begin{equation*}
F(s_1)F(s_2) = \frac{1}{L^2 s_1^2s_2^2} \,\big(D_2^{s_1}D_2^{s_2} -D_2^{s_1} - D_2^{s_2} +1  \big).
\end{equation*}
If we break $\Sigma_2$ into the four corresponding integrals, we may apply the exact same reasoning of \S \ref{Error_Double_Integral} in the first three of them to get
\begin{equation} \label{20200304_14:30}
\Sigma_2 = \frac{1}{L^2} \frac{1}{(2\pi i )^2}  \int_{\mathcal{C}}  \int_{\mathcal{C}}   \frac{\zeta(1 + s_1 + s_2) \,G(s_1, s_2)}{\zeta(1 + s_1) \zeta(1+s_2)\,s_1^2\,s_2^2} \, \ds_1 \,\ds_2 + O\left(e^{-O(\sqrt{\log D_2})}\right).
\end{equation}
Let us further work on the integral appearing in \eqref{20200304_14:30}. Fixing $s_2 \in \mathcal{C}$ we move the integral on $s_1$ to a vertical line $(\sigma_1)$ with $\sigma_1 > -\Re(s)$ for all $s \in \mathcal{C}.$ In this process we pick up two poles, one at $s_1=0$ and another one at $s_1 = -s_2$, and get
\begin{align*}
\Sigma_3  :=& \frac{1}{L^2} \frac{1}{(2\pi i )^2}  \int_{\mathcal{C}}  \int_{\mathcal{C}}   \frac{\zeta(1 + s_1 + s_2) \,G(s_1, s_2)}{\zeta(1 + s_1) \zeta(1+s_2)\,s_1^2\,s_2^2} \, \ds_1 \,\ds_2\\
 = & \frac{1}{L^2} \frac{1}{(2\pi i )^2}  \int_{\mathcal{C}}  \int_{(\sigma_1)}   \frac{\zeta(1 + s_1 + s_2) \,G(s_1, s_2)}{\zeta(1 + s_1) \zeta(1+s_2)\,s_1^2\,s_2^2} \, \ds_1 \,\ds_2 - \frac{1}{L^2} \frac{1}{(2\pi i) } \int_{\mathcal{C}} \frac{1}{s_2^2}\,\ds_2 \\
 &  \ \ \ \ \ \ \ \ \ - \frac{1}{L^2} \frac{1}{(2\pi i) } \int_{\mathcal{C}}\frac{G(-s_2, s_2)}{\zeta(1 - s_2) \zeta(1+s_2)\,s_2^4}  \,\ds_2,
\end{align*}
where we use the fact that $G(0,s_2)=1$. In the last expression above, note that the second integral is zero,
as it can be shifted to $\Re(s)\to -\infty$; however, we will let
it be for the moment. In the first integral, we may shift the contour of $s_2$ to a vertical line $(\sigma_2)$ with $ \sigma_2 >0$ (there will be a pole at $s_2=0$, but the resulting integral will be zero), and rewrite things as
\begin{align}\label{20200304_14:49}
\begin{split}
\Sigma_3 & =  \frac{1}{L^2} \frac{1}{(2\pi i )^2}  \int_{(\sigma_2)}  \int_{(\sigma_1)}   \frac{\zeta(1 + s_1 + s_2) \,G(s_1, s_2)}{\zeta(1 + s_1) \zeta(1+s_2)\,s_1^2\,s_2^2} \, \ds_1 \,\ds_2 \\
&  \ \ \ \ \ \ \  \ - \frac{1}{L^2} \frac{1}{(2\pi i) } \int_{\mathcal{C}}\left( \frac{G(-s_2, s_2)}{\zeta(1 - s_2) \zeta(1+s_2)\,s_2^4} + \frac{1}{s_2^2}\right) \,\ds_2.
\end{split}
\end{align}
The first integral in \eqref{20200304_14:49} is zero, since we can move $\sigma_1$ and $\sigma_2$ to $+\infty$ at no cost. The second integral conveniently has no pole at $s_2=0$ and we can move the contour the the imaginary axis to get
\begin{equation*}
\Sigma_3 = -  \frac{1}{2\pi L^2} \int_{-\infty}^{\infty}\left( \frac{G(-it, it)}{|\zeta(1 + it)|^2\,\,t^4} - \frac{1}{t^2}\right) \dt = \frac{\kappa}{L^2},
\end{equation*}
with $\kappa$ defined in \eqref{20200304_15:06}. This second-order contribution will be added up to the $-2\kappa/L^2$ coming from \eqref{20200227_17:35} to result in the final second-order term proposed in \eqref{20200304_12:53}. 

\subsection{The value of $\kappa$}
Let us now move to the computation of the constant $\kappa$. Since $H$ is even,
\begin{align}\label{eq:kapdef}
\begin{split}
\kappa &= \frac{1}{\pi} \left(\int_0^\varepsilon + \int_\varepsilon^1 +
\int_1^T + \int_T^\infty\right)\\
&= \frac{1}{\pi} \int_0^\varepsilon \left(t^2 - \frac{H(t)}{|\zeta(1+it)|^2}
\right) \frac{\dt}{t^4}
- \frac{1}{\pi} \int_T^\infty \frac{H(t)}{|\zeta(1+it)|^2}  \frac{\dt}{t^4} + \frac{1/\varepsilon}{\pi}\\
&  \ \ \ \ \ \  - \frac{1}{\pi}  
\int_\varepsilon^1 \frac{H(t)}{|\zeta(1+it)|^2}\frac{\dt}{t^4}-\frac{1}{\pi}\int_1^T
\frac{H(t)}{|\zeta(1+it)|^2}  \frac{\dt}{t^4}. 
\end{split}
\end{align}
Each of these integrals will be handled separately, and some of the quantitative estimates we need are presented in an Appendix at the end.

\subsubsection{The integral in the range $0\leq t\leq \varepsilon$}
We are setting apart the integral for $t$ from $0$ to $\varepsilon$ because the integrand then undergoes what is called {\em catastrophic cancellation}: when two very large terms (here: $H(t)/|\zeta(1+it)|^2 t^4$ and $1/t^2$) nearly cancel out, a na\"{\i}ve computational approach will generally result in a brutal loss in precision. Thus, we proceed to work out using a truncated Taylor series for the integrand. 
%(One might hope that this sort of procedure will at some point be automatized.)

\begin{lemma}\label{lem:plicissimus}
 Let $H(t)$ be as in \eqref{eq:Htdef}.
 Then, for $|t|\leq \frac12$,
 \begin{mysage}
   plicires = 2.56
   c2low = 1.385604
   c2high = 1.385605
   c2 = RIF(c2low, c2high);
 \end{mysage}
  \begin{equation}\label{20200229_09:49}
  1 - c_2\,t^2 \leq H(t) \leq 1 - c_2 t^2 + \sage{pf(plicires)}\,t^4,
  \end{equation}
  where 
  \begin{equation}\label{20200302_14:13}
    c_2 = \sum_p (\log p)^2/(p-1)^2 = \sage{pf(c2low)}\dotsc.
  \end{equation}
\end{lemma}
\begin{proof}
  Let us write $H(t) = \prod_p h_p(t)$, with $h_p(t)$ given by
\[h_p(t) = 1 + \frac{p^{it} + p^{-it}-2 }{\big(p^{1+it}-1\big) \big(p^{1- i t}
  - 1\big)}.\]
  We subtract and add back a term $((\log p)^2/(p-1)^2) t^2$ to get
  \[
    h_p(t) = 1 - \frac{(\log p)^2}{(p-1)^2}\,t^2  +
  2\,\frac{g_p(t)
  }{(p-1)^2 (p^2 + 1 - 2 p \cos(t \log p))},
  \]
  where
  \begin{equation*}
  g_p(t) = \big(p^2+1-2 p \cos(t \log p)\big)(\log p)^2\,\frac{t^2}{2}  -
    (p-1)^2 \big(1 - \cos (t \log p)\big).
 \end{equation*}
  Since $1-\cos x \leq x^2/2$ for any $x$, we see that $g_p(t)\geq 0$ for all $p$
  and $t$, and so
  \begin{equation}\label{eq:lowhpt}
    h_p(t) \geq 1 - \frac{(\log p)^2}{(p-1)^2}\,t^2.
    \end{equation}
  It is also easy to show that
  $1-\cos x \geq x^2/2 - x^4/4!$ for all $x$, and so
  \[\begin{aligned}g_p(t)
  &\leq \left(\frac{(p-1)^2}{4!} + \frac{p}{2}\right) (\log p)^4\,t^4.\end{aligned}\]
  Thus, since $p^2+1-2p\cos (t \log p) \geq (p-1)^2$,
  \begin{equation}\label{eq:upphpt}
    h_p(t) \leq 1 - \frac{(\log p)^2}{(p-1)^2}\,t^2  + C_p\,t^4,\end{equation}
  where $C_p = \big((\log p)^4/(p-1)^4\big) \big((p-1)^2/12 + p\big)$.
  
 \smallskip 
  
From \eqref{eq:lowhpt} we have 
\begin{equation*}
H(t)\geq M(t):=\prod_p \left(1 - \frac{(\log p)^2}{(p-1)^2}\,t^2\right)
\end{equation*}
for $|t|\leq 1$. \newb{Note that the function $f(t):= M(t)-1+c_2\,t^2$ is increasing in $0\leq t\leq 1$, where $c_2 = \sum_p (\log p)^2 /(p-1)^2$. This can be seen from the fact that $f'(0) = 0$ and $f''(t) = 2c_2 -2 \big(\sum_p a_p \prod_{p'\neq p} (1 - a_{p'}t^2)\big) + R(t)$, with $a_p = (\log p)^2 /(p-1)^2$ and $R(t)$ non-negative for $0 \leq t \leq 1$; hence $f''(t) \geq 0$ for $0 \leq t \leq 1$. We then} get $M(t)\geq 1-c_2\,t^2 $ for $|t|\leq $1, and this implies the lower bound in \eqref{20200229_09:49}.   
%  It is clear that the coefficients of the expansion
  % $\prod_p (1 - t^2 (\log p)^2/(p-1)^2) = 1 - t^2 \sum_p (\log p)^2/(p-1)^2 +
  %t^4 \sum_{p_1,p_2: p_1<p_2} (\log p_1)^2 (\log p_2)^2/((p_1-1)^2 (p_2-1)^2) - \dotsc$
  %are alternating, and that the ratio between the absolute value of two consecutive
 % coefficients beyond the first one is at most $c_2/2$, where
%  $c_2 = \sum_p (\log p)^2 (p-1)^2$. Hence, for $|t|\leq 2/c_2$,
  %\[H(t) = \prod_p (1 - h_p(t)) \geq 1 - c_2 t^2.\]
  
\smallskip

Let us now bound $H(t)$ from above. We can
write $\log H(t) = \sum_p \log h_p(t)$. Since $\log(1+x)\leq x$ for all
$x > -1$, it follows from \eqref{eq:upphpt} that
  \[\log h_p(t) \leq
  - \frac{(\log p)^2}{(p-1)^2}\,t^2  + C_p\,t^4\]
  for $|t|\leq 1$. Hence $\log H(t) \leq - c_2 t^2 + \sum_p C_p \,t^4$. Using now the fact that $\exp(-x) \leq 1 - x + x^2/2$ for $x\geq 0$ we obtain
  \begin{align} \label{20_35_18_02}
  H(t) \leq 1 - c_2\, t^2 + \sum_p C_p\, t^4 + \frac{c_2^2}{2}\, t^4
  \end{align}
provided that $t^2\leq c_2/\sum_p C_p$ \, (so that $c_2 t^2 - \sum_p C_p \,t^4 \geq 0$).
 % We let $C = c_2^2/2 + \sum_p C_p$ and obtain
 %that $H(t)\leq 1 - c_2 t^2 + C t^4$.
\smallskip

It remains to compute $c_2$ and bound $\sum_p C_p$.
There is a way to accelerate convergence following essentially the same
idea we will later use in \S \ref{subs:Heffic}; the procedure
has been worked out in \cite{HCohenPrecision}.
However, we do not need to accelerate convergence, as
brief, simple computations give acceptable results. First, observe that 
\begin{equation*}
(\log t)^4 \leq 37 \,(t-1)^{1/2}
\end{equation*}
for $t \geq 10^6$ (to see this, just square both sides and use calculus). Since $t \mapsto \frac{(\log t)^4}{(t-1)^4} \left(\frac{(t-1)^2}{12} + t \right)$ is decreasing for $t \geq 10^6$, we have
\begin{mysage}
def C(p):
  P=RBF(p)
  return RBF(((log(P)/(P-1))^4)*((P-1)^2/12+P))
SofC = C(2)
for p in [3,5..1000001]:
  if p in Primes():
    SofC+=C(p)
scc = cei(SofC,5)
\end{mysage}
\begin{mysage}
  var('t'); fgr(t) = (sqrt(t)/t^4)*(t^2/12+t+1)
  tailint = cei(RBF((37/2)*integrate(fgr(t),(t,10^6,oo))),5)
  totsumC = cei(scc+tailint,3)
\end{mysage}
\begin{align*}
 \sum_p C_p &\leq \sum_{p\leq 10^6 +1} C_p + \sum_{\stackrel{n\geq 10^6 +2}{n \ {\rm odd}}} \frac{(\log n)^4}{(n-1)^4} \left(\frac{(n-1)^2}{12} + n\right)\\
 & \leq  \sum_{p\leq 10^6} C_p + \frac{1}{2} \int_{10^6+1}^\infty \frac{(\log t)^4}{(t-1)^4} \left(\frac{(t-1)^2}{12} + t \right) \dt  \\
  & \leq  \sum_{p\leq 10^6} C_p + \frac{37}{2} \int_{10^6+1}^\infty \frac{(t-1)^{1/2}}{(t-1)^4} \left(\frac{(t-1)^2}{12} + t \right) \dt  \\
  & \leq \sage{pf(scc)} + \sage{pf(tailint)} \leq \sage{pf(totsumC)}.
  \end{align*}
For $c_2$ we proceed as follows:
  \begin{mysage}
    sprou = 1.385604464
    serr = 1.1 * 10^(-9)
    Cw = 5*10^7
  \end{mysage}
  \[
  c_2 = \sum_{p\leq \sage{pf(Cw)}} \frac{(\log p)^2}{(p-1)^2} + 
  \sum_{p> \sage{pf(Cw)}} \frac{(\log p)^2}{(p-1)^2},\]
  \[\sum_{p\leq \sage{pf(Cw)}} \frac{(\log p)^2}{(p-1)^2} 
  =\sage{pf(sprou)} + O^*(\sage{pf(serr)}).\]
\begin{mysage}
  thcplus = 0.0201384
  thcminus = 0.0239922
  lowxcond = 758711
\end{mysage}
To estimate the sum over $p>\sage{pf(Cw)}$, we will use the following estimates
on $\theta(x)= \sum_{p\leq x} \log p$, both from
\newb{\cite[Thm.~$7^*$, p. $357$]{SchoenfeldIISharper}}:
\begin{equation}\label{eq:rederor}\begin{aligned}
\theta(x) &> x - c_- \frac{x}{\log x}\;\;\;\text{for}\;
x\geq \sage{lowxcond},\;\;\;\;
\theta(x) < x + c_+ \frac{x}{\log x}\;\;\;\text{for}\;
x>1.\end{aligned}\end{equation}
for $c_- = \sage{pf(thcminus)}$, $c_+ = \sage{pf(thcplus)}$.
By integration by parts, for any $C\geq \sage{lowxcond}$,
\begin{equation}\label{eq:angjolie}\begin{aligned}\sum_{p>C} \frac{(\log p)^2}{(p-1)^2} &=
\int_C^\infty (\theta(t)-\theta(C)) \left(-\frac{\log t}{(t-1)^2}\right)' \dt
\\ &< \int_C^\infty \left(t+ c_+ \frac{t}{\log t}
-\theta(C)\right) \left(-\frac{\log t}{(t-1)^2}\right)' \dt\\
&< (c_+ +c_-) \frac{C}{(C-1)^2} + \int_C^\infty \left(t+ c_+ \frac{t}{\log t}\right)' \frac{\log t}{(t-1)^2} \,\dt,\end{aligned}\end{equation}
and it is clear that
\[\begin{aligned}
\int_C^\infty &\left(t+ c_+ \frac{t}{\log t}\right)' \frac{\log t}{(t-1)^2}\, \dt
<
\int_C^\infty \left(1+ \frac{c_+}{\log t}\right) \frac{\log t}{(t-1)^2}\, \dt
\\&< \int_C^\infty \left(\frac{\log t}{(t-1)^2} - \frac{1}{(t-1) t}
+ \frac{1+c_+}{(t-1)^2}
\right) \dt
= \frac{\log C}{C-1} + \frac{1+c_+}{C-1}.\end{aligned}\]
Hence
\begin{mysage}
  daradum = RBF((thcplus+thcminus)*Cw/(Cw-1)^2)
  daradum += RBF(log(RBF(Cw))/(Cw-1)+RBF(1+thcplus)/(Cw-1))
  daradc = cei(daradum,10)
  sprour = rou(sprou,7)
  daratot = cei(RBF(serr+daradc)+abs(RBF(sprou)-RBF(sprour)),9)
\end{mysage}
\[\begin{aligned} c_2 &= \sage{pf(sprou)} + O^*(\sage{pf(serr)}) +
O^*(\sage{pf(daradc)}) \\ &= \sage{pf(sprour)} + O^*(\sage{pf(daratot)}).
\end{aligned}\]
  %1.3856044\dotsc
\begin{mysage}
  checkme=("MISMATCH!","")[sprour-daratot>c2low and sprour+daratot<c2high]
  checkme2=("MISTAKE!","")[RBF(1/4).upper()<(RBF(c2low)/RBF(totsumC)).lower()]
  uperb = cei(RBF(RBF(totsumC)+RBF(c2high)^2/2),2)
  checkme3=("MISTAKE!","")[plicires<=uperb]
\end{mysage}
$\sage{checkme}$
Therefore for $|t|\leq \frac12$ we have $t^2\leq \frac14< c_2/\sum_p C_p$,\sage{checkme2} and using $\sum_p C_p + c_2^2/2< \sage{pf(uperb)}$ in \eqref{20_35_18_02} we obtain the desired upper bound. \sage{checkme3}
\end{proof}

\begin{corollary}\label{cor:tinyint}
  Let $H(t)$ be as in \eqref{eq:Htdef} and $c_2$ as in \eqref{20200302_14:13}. Then, for $0<\varepsilon\leq \tfrac{1}{2}$,
  \begin{mysage}
    corconst = 1
  \end{mysage}
\begin{equation*}\label{eq:ghisl}
  \int_0^\varepsilon \left(t^2 - \frac{H(t)}{|\zeta(1+it)|^2} \right)
  \frac{\dt}{t^4}
	=c\,\varepsilon+O^*(\varepsilon^3) .\end{equation*}
\begin{mysage}
  thisc = flo(c2low+stieltjes(0)^2+2*stieltjes(1),5)
  otherc = flo(c2high+stieltjes(0)^2+2*stieltjes(1),5)
  checkme = ("PROBLEM!","")[thisc==otherc]
\end{mysage}
	with $c=c_2+\gamma_0^2+2\gamma_1=\sage{pf(thisc)}\dots$, where $\gamma_n$ is the $n$-th Stieltjes constant. $\sage{checkme}$
\end{corollary}
\begin{proof}
	%{\em Truncate Taylor expansion $\zeta(1+it)=1/(i t) + \gamma + \dotsc$ with  an error term proportional to $t$.  Then     $|\zeta(1+it)|^2 = 1/t^2 - \gamma^2 + O(1)$, and so     $1/|\zeta(1+it)|^2 = t^2 + O(1)$ (give explicit version). One more term     in the expansion would be nice: it would result in an estimate     of the form $c \varepsilon + O(\varepsilon^3)$ in (\ref{eq:ghisl}), rather    than a boudn $O(\varepsilon)$.}
	
	%{\em If we do not give one more term in the expansion, then we should    weaken Lemma\ref{lem:plicissimus} and greatly shorten its proof,    so as to give $1-c_2 t^2 \leq H_p(t)\leq 1$.}
	%  For $\Re s > 1$,
	%  $\zeta(s) = \sum_n n^{-s} = 1 + \sum_{n\geq 2} n^{-s}$. By Euler-Maclaurin,
	%  \[\begin{aligned}
	%  \sum_{n\geq 2} n^{-s} &= \int_{3/2}^\infty t^{-s} + O^*\left(\frac{1}{2}
	%  \int_{3/2}^\infty |s t^{-s-1}| dt\right)\\
	%  &= \frac{(3/2)^{1-s}}{s-1} + O^*\left(\frac{|s|}{2 \sigma} (3/2)^{-\sigma}\right)
	%  .\end{aligned}\]
	%  Letting $\Re s \to 1^+$, we obtain that
	%  \[\zeta(1 + i t) = 1 + O^*\left(\frac{1}{|t|} + \frac{|t|}{3\sigma}\right).\]
	%  Hence
	By the Laurent series expansion  of $\zeta(s)$ around $s=1$,
	\[ \zeta(1+it)=\frac{1}{it}+\sum_{n=0}^\infty \frac{(-1)^n\gamma_n}{n!}(it)^n\,,\]
	where $\gamma_n$ are the Stieltjes constants. Using the bound $|\gamma_n|\leq \frac{n!}{2^{n+1}}$ (for $n\geq 1$) \cite[Lemma 4]{Lavrik}, we plainly obtain
	\[ \zeta(1+it)=\frac{1}{it}+\gamma_0-i\gamma_1\,t-\frac{\gamma_2}{2}\,t^2+\frac{i\gamma_3}{6}\,t^3+ r_1(t),\]
with
\begin{equation*}
r_1(t) =O^*\left(\frac{t^4}{16}\right)
\end{equation*}
for $|t|\leq 1$. By multiplying the above expression by its conjugate we get that
\begin{equation}\label{20200302_16:51}
|\zeta(1+it)|^2=\frac{1}{t^2}+\alpha_1+\alpha_2\,t^2+r_2(t),
\end{equation}
with $\alpha_1=\gamma_0^2+2\gamma_1$, $\alpha_2=\gamma_1^2-\gamma_0\gamma_2-\frac{\gamma_3}{3}$. Recalling that
\begin{mysage}
gamf=[]; gamc=[]; gam=[]; gamt=[];
for i in [0..3]:
  fakgam = RBF(stieltjes(i))
  gamf.append(flo(fakgam,7))
  gamc.append(cei(fakgam,7))
  gam.append(RIF(gamf[i],gamc[i]))
  if gamf[i]>=0:
    gamt.append(gamf[i])
  else:
    gamt.append(gamc[i])
\end{mysage}
$$\gamma_0=\sage{pf(gamt[0])}\dotsc,\;\;\;
\gamma_1=\sage{pf(gamt[1])}\dotsc,$$
$$\gamma_2=\sage{pf(gamt[2])}\dotsc,\;\;\;\gamma_3=\sage{pf(gamt[3])}\dotsc$$
one finds that
\begin{mysage}
  r2c = (gam[i]/2)^2+2*abs(gam[1])*abs(gam[3]/6)+abs(gam[3]/6)^3
  r2c += 2*(1+abs(gam[0])+abs(gam[1])+abs(gam[2])/2+abs(gam[3]/6))/16
  r2c += (1/16)^2
  r2cc = cei(r2c,4)
\end{mysage}
\begin{equation}\label{20200302_17:12}
r_2(t)= O^*(\sage{pf(r2cc)}\,t^4)
\end{equation}
for $|t| \leq 1$, where the constant $\sage{pf(r2cc)}$ appears as an upper bound for
\begin{align*}
\left(\left|\frac{\gamma_2}{2}\right|^2 + 2|\gamma_1|\left|\frac{\gamma_3}{6}\right| + \left|\frac{\gamma_3}{6}\right|^2 \right) + 2\cdot \left(1 + |\gamma_0| + |\gamma_1| + \left|\frac{\gamma_2}{2}\right| + \left|\frac{\gamma_3}{6}\right|\right)\cdot \frac{1}{16} + \frac{1}{16^2}.
\end{align*}	
Using \eqref{20200302_16:51} and the inequalities for $H(t)$ given in Lemma \ref{lem:plicissimus}, we obtain
\begin{equation}\label{20200302_17:13}\begin{aligned}
    \frac{(c_2+\alpha_1)+(-\sage{pf(plicires)}+\alpha_2)\,t^2+r_2(t)}{1+\alpha_1\,t^2+\alpha_2\,t^4+r_2(t)\,t^2} &\leq
    \left(t^2 - \frac{H(t)}{|\zeta(1+it)|^2}\right)\frac{1}{t^4}
  \\ &\leq
\frac{(c_2+\alpha_1)+\alpha_2\,t^2+r_2(t)}{1+\alpha_1\,t^2+\alpha_2\,t^4+r_2(t)\,t^2}
\end{aligned}\end{equation}
for $|t| \leq \frac12$.
\begin{mysage}
  alph1 = gam[0]^2+2*gam[1]; alph2 = gam[1]^2-gam[0]*gam[2]-gam[3]/3
  wob = RIF(alph1/2^2 + alph2/2^4 - RIF(r2cc)/2^8)
  checkme = ("FALSE!","")[wob.lower()>=0]
\end{mysage}
Since $\alpha_1 = \sage{pf(flo(alph1,6))}\dots$ and
$\alpha_2 = \sage{pf(flo(alph2,6))}\dots$ one has
\begin{equation}\label{20200302_17:14}
1+\alpha_1\,t^2+\alpha_2\,t^4+r_2(t)\,t^2 \geq 1
\end{equation}
for $|t|\leq \frac12$. $\sage{checkme}$ Subtracting
$(c_2 + \alpha_1)$ from the three terms in \eqref{20200302_17:13}, and using \eqref{20200302_17:12} and \eqref{20200302_17:14},
we obtain
\begin{mysage}
  radaw = RIF(plicires)-RIF(alph2)+RIF(r2cc)/4
  radaw += (c2+RIF(alph1))*RIF(alph1+alph2/4 + RIF(r2cc)/16)
  radawc = cei(radaw,2)
  checkme = ("FALSE!","")[radawc<=3*corconst]
  smallwc = cei(alph2+RIF(r2cc)/4,2)
  checkme2 = ("MISTAKE!","")[smallwc<radawc]
\end{mysage}
\begin{align} \label{19_14_19_02}
  \left(t^2 - \frac{H(t)}{|\zeta(1+it)|^2}\right)\frac{1}{t^4}=c_2 + \alpha_1+O^*(\sage{pf(radawc)} t^2)
	\end{align}
for $|t| \leq \frac12$. The constant $\sage{pf(radawc)}$ above appears as an upper bound on
\[\alpha_2 + \frac{\sage{pf(r2cc)}}{2^2}\]
and on the larger quantity\sage{checkme2}
$$(\sage{pf(plicires)} - \alpha_2) + \frac{\sage{pf(r2cc)}}{2^2}
+ (c_2 + \alpha_1)\left(\alpha_1 + \frac{\alpha_2}{2^2} +\frac{
  \sage{pf(r2cc)}}{2^4}\right)
  .$$
Naturally, \eqref{19_14_19_02} implies that 
\[\int_0^\varepsilon \left(t^2- \frac{H(t)}{|\zeta(1+it)|^2} \right) \frac{\dt}{t^4}=c\,\varepsilon+O^*\left(\frac{\sage{pf(radawc)}}{3} \varepsilon^3\right),\]
with $c=c_2+\alpha_1=\sage{pf(thisc)}\dots$, and thus we are done.
$\sage{checkme}$
\end{proof}

\subsubsection{The integral on the tail $t\geq T$}
We can easily deal with the tail integral in \eqref{eq:kapdef} by
means of the bound on $1/\zeta(1 + it)$ we will prove in the Appendix.
\begin{lemma}\label{lem:tailcont}
  For $T \geq 2$ we have
  \begin{mysage}
  futinvzet = 42.9
  consttail = cei(RBF(futinvzet)^2/27,1)
\end{mysage}
\begin{equation*}
  \int_T^\infty \frac{|H(t)|}{|\zeta(1+it)|^2} \,\frac{\dt}{t^4}\leq
  \sage{pf(consttail)}\cdot \frac{9 \log^2 T + 6  \log T + 2}{T^3}.
\end{equation*}
\end{lemma}
\begin{proof}
From \eqref{eq:Htdef} we have $0<H(t)\leq 1$ for all $t$. Using Proposition \ref{Prop10_estimates_1_over_zeta} we get
\begin{align*}
  \int_T^\infty \frac{|H(t)|}{|\zeta(1+it)|^2} \,\frac{\dt}{t^4} &\leq \int_T^\infty \frac{1}{|\zeta(1+it)|^2} \,\frac{\dt}{t^4} \leq
 \sage{pf(futinvzet)}^2 \int_T^\infty \frac{(\log t)^2}{t^4} \,\dt \\
& \leq \frac{\sage{pf(futinvzet)}^2}{27} \cdot \frac{9 \log^2 T + 6  \log T + 2}{T^3}.
\end{align*}
\end{proof}

\subsubsection{Computing $H(t)$ efficiently}\label{subs:Heffic}

The problem that remains is that of computing $H(t)$ to high accuracy
in the range $\varepsilon\leq t\leq T$, and quickly, since we are to take
a numerical integral. We should not just use the infinite product defining $H(t)$, as it converges rather slowly. We will use a trick to accelerate convergence.
The trick is well-known, and has probably been rediscovered several times;
the main idea goes back at least to Littlewood (apud \cite{WesternLittle}).
See \cite[\S 4.4.1]{HelfBook}. The idea is to express $H(t)$ as a product of zeta values times an infinite product that converges much more rapidly than $H(t)$. That infinite product can then be truncated after a moderate number of terms
at a very small cost in accuracy.

\smallskip

We can write
$H(t) = \prod_p F_2\big(p^{-1},p^{it}\big)$, where 
\[F_2(x,y) = \frac{S_2(x,y)}{(1-x y) \left(1- \frac{x}{y}\right)}\,,\]
with
\[S_2(x,y) = (1-xy) \left(1-\frac{x}{y}\right) - x^2 \left(
2 - \left(y + \frac{1}{y}\right)\right).\]
%\[S_2(x,y) = (1-xy) \left(1-\frac{x}{y}\right) - x^2 \left(
%2 - \left(y + \frac{1}{y}\right)\right).\]
We start multiplying and dividing by values of $\zeta(s)$. Clearly, 
$$H(t) = \frac{\zeta(2 + i t) \zeta(2-i t)}{\zeta(2)^2} \prod_p F_3\big(p^{-1}, p^{i t}\big),$$
where \[F_3(x,y) = \frac{S_3(x,y)}{(1-x y) \left(1 - \frac{x}{y}\right)
  (1 - x^2)^2}\,,\]
with
\[S_3(x,y) = S_2(x,y) (1 - x^2 y) \left(1-\frac{x^2}{y}\right).\]
Similarly, we may write
%Now, we iterate. Since $F_3(x,y)$ is of the form
%\[\left(1 + \left(y^2+2+\frac{1}{y^2}-2y-\frac{2}{y}\right) x^3 +
%\dotsc x^4 + \dotsc\right),\]
\begin{equation}\label{eq:foprod}
  H(t) = \left(\frac{\zeta(2 + i t) \zeta(2-i t)}{\zeta(2)^2}\right) \left(\frac{ \zeta(3+2it) \zeta(3-2it) \zeta(3)^2}{\zeta(3-it)^2 \zeta(3+it)^2} \right)  \prod_p F_4\big(p^{-1}, p^{i t}\big),\end{equation}
%{\bf Comment:} Double to check
%$$H(t) = (\zeta(2 + i t) \zeta(2-i t)/\zeta(2)^2) (\zeta(3+2it) \zeta(3-2it) \zeta(3)^2)/\zeta(3-it)^2 \zeta(3+it)^2 * \prod_p F_4(p^{-1}, p^{i t})$$
where \begin{equation}\label{eq:quietplace}
  F_4(x,y) = \frac{S_4(x,y)}{(1-x y) \left(1-\frac{x}{y}\right)
  (1 - x^2)^2\, (1-x^3 y)^2 \left(1-\frac{x^3}{y}\right)^2},\end{equation}
with
\[S_4(x,y) = S_3(x,y) (1-x^3 y^2) \left(1-\frac{x^3}{y^2}\right) (1-x^3)^2.\]
Now, the idea is to give an expression $F_4\big(p^{-1},p^{it}\big) = 1 + \varepsilon(p)$ and use it to truncate the infinite product in \eqref{eq:foprod}. Using the
definition (\ref{eq:quietplace}) of $F_4(x,y)$, we see that
\[F_4(x,y) = 1 + \frac{R_4(x,y)\,\frac{x^4}{y^4}}{(1-x y) \left(1-\frac{x}{y}\right)
	(1 - x^2)^2\, (1-x^3 y)^2 \left(1-\frac{x^3}{y}\right)^2}\,,\]
where $R_4(x,y)$ is a polynomial in $x$ and $y$. In order to estimate its value when $x = 1/p, y = p^{i t}$, we define a polynomial $Q$ in $x$ where the coefficient of $x^j$ is the sum of the absolute values of the coefficients of the monomials $x^j y^k$ in $R_4(x,y)$. By a straightforward computation,
\begin{mysage}
var('x','y')
from sage.symbolic.expression_conversions import polynomial
r.<x,y> = QQbar[]
S2(x,y) = ((1-x*y)*(1-x/y)-x*x*(2-(y+1/y)))
S3(x,y) = S2(x,y)*(1-x*x*y)*(1-x*x/y) 
S4(x,y) = S3(x,y)*(1-y*y*x*x*x)*(1-x*x*x/(y*y))*(1-x*x*x)*(1-x*x*x) 
R4(x,y) = S4(x,y)-(1-x*y)*(1-x/y)*(1-x*x)*(1-x*x)*(1-y*x*x*x)*(1-y*x*x*x)*(1-x*x*x/y)*(1-x*x*x/y)
P4 = expand(factor(simplify(expand(R4(x,y))*y^4/x^4)))
P = polynomial(P4, base_ring=QQ); Q=0
P;
for mon in P.monomials():
   Q+=abs(P.monomial_coefficient(mon))*x^mon.exponents()[0][0]
Qstr = latex(Q).replace(" \, ","")
\end{mysage}
\[
Q(x) = \sage{LatexExpr(Qstr)}.
\]
Therefore we have $F_4(p^{-1}, p^{it}) = 1 + \varepsilon(p),$ where 
$$|\varepsilon(p)|\leq \dfrac{p^{-4} \,Q(p^{-1})}{(1-p^{-1})^2 (1 - p^{-2})^2 (1-p^{-3})^4}.$$
%By \eqref{eq:satie} and the fact that $S_4(x,y)=S_4(x,y^{-1})$ we have that $\varepsilon(p)$ takes real values.
Let $C>0$ be a parameter (to be chosen later). We rewrite \eqref{eq:foprod} as
\[\begin{aligned}
H(t) = &\left(\frac{\zeta(2 + i t) \zeta(2-i t)}{\zeta(2)^2}\right) \left(\frac{ \zeta(3+2it) \zeta(3-2it) \zeta(3)^2}{\zeta(3-it)^2 \zeta(3+it)^2} \right) \\
&\cdot \prod_{p\leq C} F_4\big(p^{-1}, p^{i t}\big)\prod_{p>C}(1+\varepsilon(p)).
\end{aligned}\]
%and to show that the tail 
%$tail_C = \prod_{p>C} (1 + eps(p))$ is very close to $1$ for a reasonable value of $C$. In this way we obtain a quick way to compute $H(t)$.
\begin{mysage}
  var('t')
  D(t)=Q(x=1/t)/((1-1/t)^2*(1-1/t^2)^2*(1-1/t^3)^4)
  checkme = ("FALSE!","")[RBF(D(4)/4^4).upper()<1]
\end{mysage}
It is clear that, for $p\geq C$,
$$|\varepsilon(p)| \leq p^{-4} D(C)\leq C^{-4} D(C),$$
where $D(t) := Q(t^{-1})/\big((1-t^{-1})^2 (1 - t^{-2})^2 (1-t^{-3})^4\big).$ We write $\delta = \delta(C) = C^{-4} D(C)$.
Since $D(C)$ is a decreasing function of $C$, so is
$\delta(C)$. Thus,
$\delta(C)\leq \delta(4)<1$ for $C\geq 4$. $\sage{checkme}$
By the mean value theorem,
$$
|\log(1+\varepsilon(p))|\leq |\varepsilon(p)|\max_{|\xi|\leq\delta}\bigg|\dfrac{1}{1+\xi}\bigg|=|\varepsilon(p)|\,\dfrac{1}{1-\delta}
$$
for $p\geq C$. Therefore 
$$\Bigg|\log \prod_{p>C} \big(1 + \varepsilon(p)\big)\Bigg| = \Bigg|\sum_{p>C} \log\big(1+\varepsilon(p)\big)\Bigg| \leq \dfrac{1}{1-\delta}\,\sum_{p>C} |\varepsilon(p)| \leq \dfrac{D(C)}{1-\delta}\, \sum_{p>C} \dfrac{1}{p^4}.$$

To estimate $\sum_{p>C} 1/p^4$, we will use the upper bound on $\theta(x)$ in
(\ref{eq:rederor}), together with the following lower bound
from \cite[Cor.~$2^*$]{SchoenfeldIISharper}:
\begin{mysage}
  lowerxcond = 67
  thc0minus = 6/7
\end{mysage}
\[\theta(x) > x - \sage{thc0minus} \frac{x}{\log x}\;\;\;\text{for}\;
x\geq \sage{lowerxcond}.\]
We proceed much as in (\ref{eq:angjolie}):
by integration by parts, 
\[\begin{aligned}
\sum_{p>C} \frac{1}{p^4} &= \int_{C}^{\infty} (\theta(t)-\theta(C))
\left(- \frac{1}{t^4 \log t}\right)' \dt \\ &<
\int_{C}^{\infty} \left(t + c_+ \frac{t}{\log t} -\theta(C)\right) 
\left(- \frac{1}{t^4 \log t}\right)' \dt \\ &<
\frac{c_+ + \sage{thc0minus}}{C^3 \log^2 C} +
\int_{C}^{\infty} \left(t + c_+ \frac{t}{\log t}\right)'
\frac{\dt}{t^4 \log t}\end{aligned}\]
and it is easy to see that
\[
\int_{C}^{\infty} \left(t + c_+ \frac{t}{\log t}\right)'
\frac{\dt}{t^4 \log t} <
\int_{C}^{\infty} \left(1 + \frac{c_+}{\log t}\right)\frac{\dt}{t^4 \log t}\\
< \frac{1}{3 C^3 \log C},\]
since $c_+<1/3$. So, for $C\geq \sage{lowerxcond}$, 
$$\bigg|\log\prod_{p>C} \big(1 + \varepsilon (p)\big)\bigg|  \leq
D(C) \cdot \frac{\frac{1}{3} + \frac{c_++\sage{thc0minus}}{\log C}
}{(1-\delta) C^{3} \log C}.$$
Here $D(C)$ tends rapidly to the constant coefficient of $Q$ (that is, $20$) when $C\to \infty.$ By $|e^x-1|\leq e^{|x|}-1$, we see that
$$\bigg|\prod_{p>C} \big(1 + \varepsilon (p)\big)^{-1} -1\bigg|  \leq e^{\rho(C)}-1,
$$
where $\rho(C)= D(C)\cdot (1/3+(c_++\sage{thc0minus})/\log C)/(1-\delta)C^{3} \log C$.
Since $|H(t)|\leq 1$, we conclude that
\begin{equation}\label{eq:compred}\begin{aligned}
&\left(\frac{\zeta(2 + i t) \zeta(2-i t)}{\zeta(2)^2}\right) \left(\frac{ \zeta(3+2it) \zeta(3-2it) \zeta(3)^2}{\zeta(3-it)^2 \zeta(3+it)^2} \right)  \prod_{p\leq C} F_4(p^{-1}, p^{i t}) \\ &=
H(t) \prod_{p>C}(1+\varepsilon(p))^{-1} = \!H(t) + O^*\big(e^{\rho(C)}-1\big).
\end{aligned}\end{equation}
We will denote the product on the left-hand side of \eqref{eq:compred}
by $H_C(t)$. Thus $H(t) = H_C(t) +  O^*\big(e^{\rho(C)}-1\big)$.

\smallskip

Here is a table with some values of the quantities we have just discussed.
\begin{mysage}
delta(t) = D(t)/t^4
rho(t) = (D(t)/(1-delta(t)))*(1/3+(RBF(thcplus)+thc0minus)/log(t))/(t^3*log(t))
\end{mysage}

	\begin{center}
		\begin{tabular}{|c|c|c|c|c|}
		  \hline $C$ & $D(C)$ & $\delta$ & $\rho(C)$ & $\exp(\rho(C))-1$\\
                  \hline
                  $250$ & $\sage{pf(flo(D(250),5))}\dotsc$ & $\leq \sage{pf(cei(delta(250),10))}$ & $\leq \sage{pf(cei(rho(250),10))}$ & $\leq \sage{pf(cei(exp(rho(250))-1,10))}$\\
                  $750$ & $\sage{pf(flo(D(750),6))}\dotsc$ & $\leq \sage{pf(cei(delta(750),10))}$ & $\leq \sage{pf(cei(rho(750),12))}$ & $\leq \sage{pf(cei(exp(rho(750))-1,12))}$\\
                  $3000$ & $\sage{pf(flo(D(3000),5))}\dotsc$ & $\leq \sage{pf(cei(delta(3000),14))}$ & $\leq \sage{pf(cei(rho(3000),13))}$ & $\leq \sage{pf(cei(exp(rho(3000))-1,13))}$\\
			\hline
		\end{tabular}
	\end{center}
	\label{default}

        \smallskip
        \noindent
{\em Remark.}        D.\,Zagier suggests the following variant, which would also
        be applicable to other products like $H(t)$. We can repeat the above
        procedure {\em ad infinitum}, expressing $H(t)$ as an infinite
        product of values of the form $\zeta(a + i b t)$, $a\geq 2$,
        $|b|<a$, $a,b\in \mathbb{Z}$. In order to ensure absolute convergence,
        we may choose to work with an infinite product of
        values of
        \[\zeta_{>C}(s) = \zeta(s) \prod_{p\leq C} (1 - p^{-s}),\]
for some sufficiently large $C$,
        and multiply in the end by $\prod_{p\leq C} F_2(p^{-1},p^{i t})$,
        We then obtain an expression of the form
        \begin{equation}\label{eq:infiproH}
          H(t) =
        \prod_{p\leq C} F_2(p^{-1},p^{i t}) \cdot
        \prod_{a\geq 2} \prod_{|b|<a} \zeta_{>C}(a + i b t)^{\alpha_{a,b}},
        \end{equation}
        where $\alpha_{a,b}$ can be determined recursively and bounded
        fairly easily. (In the
        particular case of our
        product $H(t)$, a closed expression for $\alpha_{a,b}$ in terms of the
        the
        series expansion of $(1+x+\sqrt{(1-x) (1 + 3 x)})/2$
        is also possible.) One can bound the tail
        $\prod_{a>A} \prod_{|b|<a}$ of the double product in (\ref{eq:infiproH})
        using $|\zeta_{>C}(a+i b t)|\leq |\zeta_{>C}(a)|$ and the following easy
        bound: for $C$ an integer,
        \[|\zeta_{>C}(a)|\leq
1 + \sum_{n\; \text{odd},\; n>C} \frac{1}{n^a} \leq
        1 + \frac{1}{2} \int_C^\infty t^{-a} \dt =
        1 + \frac{C^{1-a}}{2 (a-1)},\]
        where we use the convexity of $t\mapsto t^{-a}$.
        In the case of our product $H(t)$, P. Moree
        points out that $\alpha_{a,b}$
        grows slowly enough that taking $C=4$ is sufficient to ensure
        absolute convergence; one can of course also take a larger $C$.
%Thus the product over the primes $p\leq 20000$ is very good. In fact, the error term will come from other terms. Indeed, for $t$ big enough, to compute the values of zeta will take more time than to compute the product. \\

%\textcolor{red}{To be concluded}

        \subsubsection{Conclusion}

        Let us first compute the integral from $\varepsilon$ to $1$, setting
        $\varepsilon=2\cdot 10^{-3}$. It makes sense to split the integral
        into (at least) two parts, since we will need to approximate
        $H(t)$ to different precisions:
        \[\begin{aligned}
        \int_\varepsilon^1 \frac{H(t)}{|\zeta(1+it)|^2} 
        \frac{\dt}{t^4} &=
        \int_{2\cdot 10^{-3}}^{2\cdot 10^{-1}} \frac{H(t)}{|\zeta(1+it)|^2} 
        \frac{\dt}{t^4} +
                \int_{2\cdot 10^{-1}}^{1} \frac{H(t)}{|\zeta(1+it)|^2}
                \frac{\dt}{t^4}
                .\end{aligned}\]
        By our discussion above,
\begin{mysage}
  integ01r = 3.20641404 
  integ01err = 3.9 * 10^(-9)
  erm01 = 3.3468 * 10^(-9)
  merr01 = 3.85768
  toterr01 = cei(RBF(integ01err)+RBF(cei(RBF(erm01)*RBF(merr01),12)),9)
\end{mysage}
\[\begin{aligned}
\int_{2\cdot 10^{-1}}^{1} \frac{H(t)}{|\zeta(1+it)|^2} 
\frac{\dt}{t^4} &=
\int_{2\cdot 10^{-1}}^1 \frac{H_{750}(t)}{|\zeta(1+it)|^2} 
\frac{\dt}{t^4} \\ &+
O^*(e^{\rho(750)}-1)
\int_{2\cdot 10^{-1}}^1 \frac{1}{|\zeta(1+it)|^2} 
\frac{\dt}{t^4} \\ &=
                \sage{pf(integ01r)} + O^*(\sage{pf(integ01err)}) \\ &+
                O^*(\sage{pf(erm01)})\cdot O^*(\sage{pf(merr01)}) \\ &=
                \sage{pf(integ01r)} + O^*(\sage{pf(toterr01)}),
\end{aligned}\]
\begin{mysage}
  integ02r = 494.69534269
  integ02err = 1.44 * 10^(-8)
  erm02 = 4.1011 * 10^(-11)
  merr02 = 494.96295
  toterr02 = cei(RBF(integ02err)+RBF(cei(RBF(erm02)*RBF(merr02),12)),9)
\end{mysage}
\[\begin{aligned}\int_{2\cdot 10^{-3}}^{2\cdot
  10^{-1}} \frac{H(t)}{|\zeta(1+it)|^2} 
        \frac{\dt}{t^4} &=
        \int_{2\cdot 10^{-3}}^{2\cdot 10^{-1}} \frac{H_{3000}(t)}{|\zeta(1+it)|^2} 
        \frac{\dt}{t^4} \\ &+
        O^*(e^{\rho(3000)}-1)
                \int_{2\cdot 10^{-3}}^{2\cdot 10^{-1}} \frac{1}{|\zeta(1+it)|^2} 
                \frac{\dt}{t^4} \\ &=
                \sage{pf(integ02r)} + O^*(\sage{pf(integ02err)}) \\ &+
                O^*(\sage{pf(erm02)})\cdot O^*(\sage{pf(merr02)}) \\ &=
                \sage{pf(integ02r)} + O^*(\sage{pf(toterr02)}),
        \end{aligned}\]
        where we perform rigorous numerical integration by means
        of the ARB ball-arithmetic package. Hence, in total,
\begin{mysage}
  integs = integ01r+integ02r; integr = rou(integs,8)
  integerr = cei(RBF(abs(integr-integs))+RBF(toterr01)+RBF(toterr02),9)
\end{mysage}
        \[\int_{2\cdot 10^{-3}}^{1} \frac{H(t)}{|\zeta(1+it)|^2} \frac{\dt}{t^4} =
        \sage{pf(integr)}+O^*(\sage{pf(integerr)}).\]

\begin{mysage}
  T = 7500
\end{mysage}

        It remains to compute the integral
        $\int_1^T \frac{H(t)}{|\zeta(1+it)|^2} \frac{\dt}{t^4}$
        for a reasonable value of $T$. We choose $T=\sage{T}$.
\begin{mysage}
  T0=200
  sintb = 7.1035 * 10^(-8)
\end{mysage}
First of all: interval arithmetic gives us (among other things)
an upper bound on the maximum of a real-valued
function (such as $1/|\zeta(1+it)|^2 t^4$) on an interval (say,
$\lbrack r, r + 1/100\rbrack$. In this way, letting $r$
range over $\frac{1}{100}\mathbb{Z} \cap [\sage{T0},\sage{T}]$ and
then summing, we get that
\[\int_{\sage{T0}}^{\sage{T}}  \frac{1}{|\zeta(1+it)|^2} \frac{\dt}{t^4}
\leq \sage{pf(sintb)}.\]
Of course,
\[\left|\int_{\sage{T0}}^{\sage{T}}
 \frac{H(t)}{|\zeta(1+it)|^2} \frac{\dt}{t^4}\right|
 \leq \int_{\sage{T0}}^{\sage{T}}
 \frac{1}{|\zeta(1+it)|^2} \frac{\dt}{t^4}.\]
\begin{mysage}
  integlr = 0.19345589
  integlerr = 9.58 * 10^(-8)
  erml = 1.153 * 10^(-7) 
  brut = 0.44903
  errl = cei(RBF(erml)*RBF(brut),10)
  errtl = integlerr+errl
  finerr = cei(errtl+cei(sintb,12),10)
\end{mysage}   
By rigorous numerical integration in ARB,
 \[\int_1^{\sage{T0}} 
 \frac{H_{250}(t)}{|\zeta(1+it)|^2} \frac{\dt}{t^4} =
 \sage{pf(integlr)} + O^*(\sage{pf(integlerr)}).\]
 Much as before, we have an additional error term
 \[\begin{aligned}O^*(e^{\rho(\sage{T0})}-1)\cdot \int_1^{\sage{T0}} \frac{1}{|\zeta(1+it)|^2} \frac{\dt}{t^4}
 &= O^*(\sage{pf(erml)}) \cdot O^*(\sage{pf(brut)}) \\ &= O^*(\sage{pf(errl)}).
 \end{aligned}\]
 Hence
 \[\int_1^{\sage{T0}}  \frac{H(t)}{|\zeta(1+it)|^2} \frac{\dt}{t^4} =
 \sage{pf(integlr)} + O^*(\sage{pf(errtl)}).\]

 Putting our bound on the integral from $\sage{T0}$ to $\sage{T}$
 in the error term, we obtain that
 \[
   \int_1^{\sage{T}}  \frac{H(t)}{|\zeta(1+it)|^2} \frac{\dt}{t^4} =
 \sage{pf(integlr)} + O^*(\sage{pf(finerr)})\]
 for $T=\sage{T}$. By Lemma \ref{lem:tailcont},
\begin{mysage}
  logT = RBF(log(RBF(T)))  
  tailcnt = RBF(RBF(consttail)*(9*logT^2+6*logT+2)/RBF(T)^3)
  tailb = cei(tailcnt,11)
\end{mysage}
\[ \int_{\sage{T}}^\infty \frac{|H(t)|}{|\zeta(1+it)|^2} \,\frac{\dt}{t^4}\leq
\sage{pf(consttail)}\cdot \frac{9 \log^2 \sage{T} + 6  \log \sage{T} + 2}{\sage{T}^3} \leq \sage{pf(tailb)}.
\]
\begin{mysage}
  vareps = 2/1000
  dumberr = 10^(-8)
  sillyerrb = n(vareps^3+dumberr)
  newapp = vareps*thisc + 10^(-8)
\end{mysage}
By Corollary \ref{cor:tinyint},
\[	\int_0^{2\cdot 10^{-3}} \left(t^2-\frac{H(t)}{|\zeta(1+it)|^2}\right) \frac{\dt}{t^4} = \sage{pf(newapp)} + O^*(\sage{pf(dumberr)})
+ O^*(\sage{pf(n(vareps^3))}).\]
Going back to (\ref{eq:kapdef}), we obtain that
\begin{mysage}
  maint = RBF((1/vareps+RBF(newapp)-RBF(integr)-RBF(integlr))/pi)
  maintr = rou(maint,8)
  mainterr = cei(abs(maint-RBF(maintr)),10)
  enderr = RBF(mainterr + (integerr+finerr+tailb+sillyerrb)/pi)
  enderrc = cei(enderr,9)
\end{mysage}
\[\begin{aligned}\kappa &= \frac{1}{\pi}
\left(\sage{1/vareps} + \sage{pf(newapp)}
- \sage{pf(integr)} - \sage{pf(integlr)}\right)\\ &+ \frac{1}{\pi}
O^*\left(\sage{pf(integerr)}+\sage{pf(finerr)}+\sage{pf(tailb)}
+\sage{pf(sillyerrb)}\right)\\
&= \sage{pf(maintr)} + O^*(\sage{pf(enderrc)}).\end{aligned}\]
\begin{mysage}
  remmain = maintr-flo(maintr,6)
  checkme=("MISMATCH!","")[flo(maintr,6)==kappa]
  checkme2=("MISTAKE!","")[remmain>enderrc and remmain+enderrc<0.000001]
\end{mysage}
{\color{red} $\sage{checkme}$} {\color{red} $\sage{checkme2}$}

\section{Concluding remarks}
%D_2^2\gg N
%D_1 \ll \sqrt{N}\ll D_2
%greater v; v=2
%Brun-Titchmarsh {\em Side remark.} It should be clear that,
%explicit; Sebastian does a nice job for real part such and such

We conclude by briefly mentioning some classical and recent applications of quadratic sieves and outlining a few potential directions for further research.
  
  \subsection{Uses of quadratic sieves}\label{subs:backmot}
  %{\em Sieves in their traditional setting.} In order to provide 
  \subsubsection{Classical framework}
  The classical application of sieves -- from which they take
  their name -- consists in estimating the number of integers that are
  excluded from certain congruence classes modulo $p$ for all primes $p$
  in a set $\mathscr{P}$.
  For instance, we may want to count integers that are excluded from
  the congruence class $0 \bmod p$ for every $p\in \mathscr{P}$, that is,
  integers coprime to all $p\in \mathscr{P}$.
  It is clear that the expression
  $\left(\sum_{d|n} \mu(d) \rho(d)\right)^2$
  equals $1$ if $n\not\equiv 0 \bmod p$ for all $p<D_2$.
  Being a square, it is also non-negative for $n$ arbitrary.
  Hence the sum $S_\rho$ in (\ref{eq:sums}) is an upper bound on the number
  of integers $n\leq N$ without prime factors $p<D_2$, and thus it is also
  an upper bound on the number of primes between $D_2$ and $N$. Of course one can obtain precise estimates for that number of primes
  by analytic means instead. What is remarkable about sieves is their flexibility.
  For instance, we may decide that we want to count primes in
  an arithmetic progression $a+m\mathbb{Z}$, rather than among integers as
  a whole. Then we are considering
  \[ \mathop{\sum_{n\in a+ m \mathbb{Z}}}_{1\leq n\leq N}
  \Bigg(\sum_{d|n} \mu(d) \rho(d)\Bigg)^2\]
  and the analysis goes almost exactly as it will for $S_\rho$; the upper bound
  we then obtain is a form of the Brun-Titchmarsh theorem, which gives
  us information even when $m$ is close in size to $N$ (as a straightforward
  analytic approach by means of $L$-functions cannot).
  It is also through sieves that we can obtain upper bounds on the number
  of twin primes in an interval, and so forth.

  \subsubsection{Further uses of sieves.} \label{Envel_Sieves}
  Sieves, used on their own, have their limits (the {\em parity problem}).
  Great progress has been made in the last 20 years or so by combining
  sieves with other techniques. In particular, there is what is now
  called {\em enveloping sieves} (after Hooley and Ramar\'e). We are using
  a sieve as an {\em enveloping sieve} when we use the expression
  $(\sum_{d|n} \mu(d) \rho(d))^2$, not directly to count primes, but
  as a weight, in order to bias $n$ towards being a prime. Then we can work
  by other means with those weighted integers $n$. This approach achieved a remarkable success in the work of Goldston-Pintz-Y{\i}ld{\i}r{\i}m.
In their work,
what we find is a generalization (dimension $>1$)
of the kind of sieve we consider.
 A similar approach is that of Green-Tao \cite{zbMATH05578697}, who (relying on Goldston and Y{\i}ld{\i}r{\i}m's analysis) use the weight $(\sum_{d|n} \mu(d) \rho(d))^2$ as a majorant within which primes are of positive density, so to speak; then they are able to adapt techniques developed for sets of positive density within the integers.
 
 \subsubsection{Quadratic sieves, appearing uninvited.} Sums such as $S_\rho$
 can also appear
 naturally when we are working on other problems, without
 any intention to sieve. Say that, as often
 happens in analytic number theory, we use Vaughan's identity, followed
 by Cauchy-Schwarz. Then we have a sum
 \begin{equation}\label{eq:rendremo}
   \sum_{1\leq n\leq N} \Bigg(\mathop{\sum_{d|n}}_{d>D} \mu(d)\Bigg)^2
 \end{equation}
 to bound. See \cite{MR691960} and \cite{BTal}, which prove
 an asymptotic of the form $(c+o(1)) N$ for (\ref{eq:rendremo}). We may decide to do one better, and use a version of Vaughan's identity with
 a smooth cutoff $\rho$.
 Then we must bound a sum
 \begin{equation}\label{eq:garnacho}
   \sum_{1\leq n\leq N} \Bigg(\sum_{d|n} (1-\rho(d)) \mu(d)\Bigg)^2,
   \end{equation}
 which, by M\"obius inversion, equals our sum $S_\rho$ plus a constant
 term $-1$. This is the situation that gives rise to the use
 of a quadratic sieve in \cite{HelfBook}.\footnote{The original version of
   the proof of the ternary Goldbach conjecture (\cite{Helf}, \cite{HelfTern})
   did not use
   a sieve, relying instead on a detailed explicit study of
   (\ref{eq:rendremo}); what is at stake here is an improvement in the
   original proof, resulting in sharper bounds.}
 Another application is that in \cite{sedunova2019logarithmic}, where a sum of type
 (\ref{eq:garnacho}) arises in the context of sharpening the
 Bombieri-Vinogradov inequality (the same application
 motivated \cite{MR691960}).
  
 \subsubsection{Sieves as weights for coefficients of Dirichlet series.}\label{subs:siedirich}
 A quadratic sieve also appears
 in the study of Linnik's problem \cite{zbMATH03590364}: there, a sum
 of squares of $\sum_{d|n} \rho(d) \mu(d)$ appears as a result of
 the application of the duality principle behind the large sieve. Then
 $\rho$ is chosen so as to better bound the number of
 zeros of $L(s,\chi)$ close to $s=1$; a sharp truncation $\rho$ would not
 be sufficient. The smaller the sum $S_\rho$ is, the better $\rho$ is for this
 purpose.
 Such was the motivation for Graham's work on $S_\rho$ in \cite{zbMATH03593672}. Actually, even Selberg's introduction of the kind of sieves we are
 studying has its roots in his earlier work \cite{MR0010712}
 on zeros of the zeta function. A detailed discussion
 can be found in \cite[\S 7.2]{MR2647984}.
 There has been further use of $\rho_{D_1,D_2,h}$ in the context of mollifiers;
 see, e.g., \cite{zbMATH01965437} (in particular, (2.8) therein) and
 subsequent work.

\subsection{Future directions}\label{subs:futdir}
\subsubsection{Broader ranges for parameters}
There are applications for which it is necessary to cover
precisely the cases $D_2\ggg \sqrt{N}$ (for the one-parameter sieve)
and $D_2\geq D_1\ggg \sqrt{N}$ (for the two-parameter sieve); these cases
are inaccessible to most small sieves. Such is the case both in
\cite{zbMATH03593672}, which allowed an improved bound on Linnik's
constant (\cite{zbMATH03565098}, \cite{zbMATH03590364}), and in
\cite{HelfBook}. It does seem possible to adapt the analysis
here to prove the
the optimality of $h=h_0$
in the cases $D_2\ggg \sqrt{N}$ and $D_2\geq D_1\ggg \sqrt{N}$, in the
sense of Corollary \ref{Prop1}. One may also want to deal with fully
general $D_1$, $D_2$, that is, one could aim to give bounds that are valid for
all $D_1$, $D_2$, and good when $D_1\ggg 1$ and $N/D_2\ggg 1$ (or $D_1=1$ and
$N/D_2\ggg 1$).
The case $D_1\ll \sqrt{N}\ll D_2$ is delicate.
Here the idea at the end of \cite{BTal} might be useful.

\subsubsection{Combining a quadratic sieve with a preliminary sieve.}
It is common to combine sieves with a na\"ive sieve that takes
care of small primes. In our case, we would need to study
\[S_{v,\rho} = \mathop{\sum_{1\leq n\leq N}}_{(n,v)=1}
\left(\sum_{d|n} \mu(d) \rho(d)\right)^2\] for small $v$.
This more general sum is in fact studied in \cite{HelfBook}, \cite{SZThesis}
and \cite{SZnew}, with a second-order term being worked explicitly for $v=2$.
It would be worthwhile to do the same for the analysis in the present paper. As an example of how even just the case $v=2$ is helpful, consider the problem
of proving Brun-Titchmarsh, in the strong form in
\cite[(1.10)]{MR0374060}.
Then it is important to know the second-order term
in \eqref{20200304_12:53} -- and in fact, while it is good that it is negative,
it does not seem to be quite enough. However, a version of
\eqref{20200304_12:53} for $v=2$ is in fact sufficient for proving
\cite[(1.10)]{MR0374060}, at least (to use the notation there)
for $y/k$ larger than a constant (much as in \cite[(22.15')]{Sellec}, or
\cite[(1.11)]{MR0374060}).
We already know the constant for $v=2$ in that case, thanks to
\cite{SZThesis}.

\subsubsection{Explicit bounds.} The bounds in \cite{HelfBook}, \cite{SZThesis}
and \cite{SZnew} are all fully explicit
(with an error term qualitatively
larger than that in
\eqref{20200304_12:53} or \eqref{20200227_17:02}). It would be desirable to have
explicit bounds for the error terms in \eqref{20200304_12:53} and
\eqref{20200227_17:02} resulting from our approach. In the past, treating
sums involving $\mu(n)$ by complex analysis was sometimes considered
unfeasible, due in part
to the absence of good bounds on $1/\zeta(s)$ inside the critical strip.
Since we give a usable bound in Proposition \ref{Prop10_estimates_1_over_zeta}, and
since our integrands decay reasonably rapidly, aiming at good explicit error
terms through our approach would in fact seem realistic. Again, simply as an illustration, note that explicit bounds are needed if
we want to reprove \cite[(1.10)]{MR0374060}
(that is, the Brun-Titchmarsh inequality in its modern form)
without the assumption that
$y/k$ 
be larger than a constant, or even just to make that constant explicit.
Of course there are plenty of other applications of explicit bounds, with
their use in \cite{HelfBook} being an example.

The chapter of \cite{HelfBook} on the quadratic sieve
is now being revised, with the aim of giving explicit estimates in the range $D_2\geq D_1 \ggg \sqrt{N}$ by complex-analytic means, following a strategy
inspired in part by the present paper. Some of the authors of this paper
are also working on new explicit estimates for sums of $\mu(n)$, also based
on a  complex-analytic approach. One of the novelties there consists in
foregoing the direct application of $L^\infty$
bounds like Proposition \ref{Prop10_estimates_1_over_zeta} 
in favor of $L^2$-bounds on the line $\Re s = 1$ (themselves relying in part
on $L^\infty$ bounds on the critical strip).

Of course, once the range
$D_2\geq D_1\ggg \sqrt{N}$ goes through, we expect that it will also be quite feasible to make matters fully explicit in the range $D_2\lll \sqrt{N}$ we have treated here (as it is somewhat more straightforward).

\section*{Appendix A. Explicit estimates on $\zeta(s)$}
Here we prove some quantitative estimates for the Riemann zeta-function that may be of independent interest. We remark that the estimates in Propositions \ref{Prop8_Bounds_zeta} and \ref{Prop10_estimates_1_over_zeta} are not qualitatively the best available (see e.g. \cite[Chapter VI]{MR882550}), but estimates
of this form are sufficient for our purposes. What is important, for practical
purposes, is to have an explicit bound on $1/\zeta(s)$ with a reasonable
constant, as in Prop.~\ref{Prop10_estimates_1_over_zeta}.
\begin{mysage}
  T1=500
  helpconst = 0.14
\end{mysage}
\begin{proposition}\label{Prop8_Bounds_zeta}
  For $1\leq \sigma\leq 2$ and $t\geq \sage{T1}$ we have
	$$
	|\zeta(\sigma+it)|< \log t -\sage{pf(helpconst)}.
	$$	 
\end{proposition}
\begin{proof} We follow the idea of Backlund in \cite{backlund1916}. Let $s=\sigma+it$ such that  $1< \sigma\leq 2$ and $t\geq \sage{T1}$. For $N\geq 2$, by \cite[Eq. (8)]{backlund1916} we have the representation
	\begin{align} \label{1_4_6_20}
	\zeta(s)=\displaystyle\sum_{n=1}^{N-1}\dfrac{1}{n^s} + \dfrac{1}{2N^s}+\dfrac{N^{1-s}}{s-1}+\dfrac{s}{12N^{s+1}}-\dfrac{s(s+1)}{2}\int_{N}^{\infty}\dfrac{\varphi^*(u)}{u^{s+2}}\,\du\,,
	\end{align}
	where $\varphi^*(u)$ is the periodic function obtained by the extension
        of the polynomial $\varphi(u)=u^2-u+1/6$ on the interval $[0,1]$.
Using the estimate in \cite[p.~361]{backlund1916} we have 
		\begin{align*} 
	\bigg|\displaystyle\sum_{n=1}^{N-1}\dfrac{1}{n^s} + \dfrac{1}{2N^s}\bigg|\leq \displaystyle\sum_{n=1}^{N-1}\dfrac{1}{n} + \dfrac{1}{2N} < \log N + \gamma,
	        \end{align*}
\begin{mysage}
   almostone = 1.001
\end{mysage}
where $\gamma = \sage{pf(flo(euler_gamma,7))}\dots$ is the Euler's constant. We bound the term $N^{1-s}/(s-1)$ by $1/t$. Also, if we write $\alpha_1=|s/{t}|$ and $\alpha_2=|s(s+1)/t^2|$ we have $\alpha_1, \alpha_2< \sage{pf(almostone)}$. Finally, using the bound $|\varphi(u)|\leq 1/6$ on $[0,1]$, it follows that 
	\begin{align*}
	\bigg|\int_{N}^{\infty}\dfrac{\varphi^*(u)}{u^{s+2}}\du\bigg|\leq \dfrac{1}{6}\int_{N}^{\infty}\dfrac{\du}{u^{3}} = \dfrac{1}{12N^2}.
	\end{align*}
\begin{mysage}
  ko1 = cei(RBF(almostone)/12,3)
  ko2 = cei(RBF(almostone)/24,3)
\end{mysage}
	Therefore, combining these estimates in \eqref{1_4_6_20} we get
	\begin{align} \label{1_21_6_20}
	|\zeta(\sigma+it)|< \log N + \gamma + \dfrac{1}{t}+\dfrac{\sage{pf(ko1)}\,t}{N^{2}}+\dfrac{\sage{pf(ko2)}\,t^{2}}{N^{2}}.
	\end{align}
	Now, let $\lambda>0$ be a parameter and define $
	N=\lfloor\frac{t}{\lambda}\rfloor+1$. Then,
	$$
	\log N < \log\bigg(\dfrac{t}{\lambda}+1\bigg)= \log t -\log\lambda + \log\bigg(1+\dfrac{\lambda}{t}\bigg) < \log t -\log\lambda + \dfrac{\lambda}{t}.
	$$
	Recalling that $1/N<\lambda/t$ and $t\geq \sage{T1}$, we obtain in \eqref{1_21_6_20} that
\begin{mysage}
  sill = (1/T1).n()
  derc = cei(RBF(ko2)+RBF(ko1)/T1,3)
\end{mysage}
\begin{align*}
	|\zeta(\sigma+it)|&<(\log t -\log\lambda + \sage{pf(sill)}\lambda ) + \gamma + \sage{pf(sill)} + \sage{pf(derc)}\,\lambda^{2}.
\end{align*}
\begin{mysage}
  mylamb = 3.3983
  lamb = RBF(mylamb)
  mybc = flo(log(lamb)-RBF(sill)*lamb-euler_gamma-RBF(sill)-RBF(derc)*lamb^2,3)
  checkme=("MISMATCH!","")[helpconst<=mybc]
\end{mysage}          
	Optimizing over $\lambda>0$ ($\lambda \approx \sage{pf(mylamb)}$),
        we obtain that $|\zeta(\sigma+i t)| < \log t - \sage{pf(mybc)}$.
        \sage{checkme}
\end{proof}

\begin{proposition}\label{Prop10_estimates_1_over_zeta}
  For $t\geq 2$ we have
  \begin{mysage}
    invzetbound = 42.9
    checkme=("MISMATCH!","")[futinvzet==invzetbound]
  \end{mysage}
$$
	\bigg|\dfrac{1}{\zeta(1+it)}\bigg|< \sage{pf(invzetbound)} \log t.
	$$
        $\sage{checkme}$
\end{proposition}	 
For comparison: Table 2 in \cite{trudgian2015} gives the bound
$|1/\zeta(\sigma + i t)|\leq 1900 \log t$ for $|t|\geq 132.16$ and 
$\sigma\geq 1 - 1/12 \log t$. We focus on the case $\sigma=1$.
\begin{proof}
\begin{mysage}
  trudbou = 40.14
\end{mysage}
  First we suppose that $t\geq \sage{T1}$. Let $d>0$ be a parameter (to be properly chosen later). From \cite[Table 2]{trudgian2015} the estimate 
	\begin{align*}
	\bigg|\dfrac{\zeta'(\sigma+it)}{\zeta(\sigma+it)}\bigg|\leq \sage{pf(trudbou)}\,\log t
	\end{align*}
holds for $\sigma\geq 1$. Then, 
	\begin{align} \label{21_19_6_20}
	\begin{split}
	\log\bigg|\dfrac{1}{\zeta(1+it)}\bigg|& =-\Re\log\zeta(1+it) \\
	& = -\Re\log\zeta\bigg(1+\dfrac{d}{\log t}+it\bigg) + \int_{1}^{1+\frac{d}{\log t}}\Re\,\dfrac{\zeta'}{\zeta}(\sigma+it)\,\d\sigma \\
	& \leq  -\log\bigg|\zeta\bigg(1+\dfrac{d}{\log t}+it\bigg)\bigg| + \sage{pf(trudbou)}\,d.
	\end{split}
	\end{align} 	
	On the other hand, we recall the classical estimate \cite[Eq. (2), Chapter 13]{MR1790423}
	\begin{align*}  
	\zeta^3(\sigma)|\zeta^4(\sigma+it)\,\zeta(\sigma+2it)|\geq 1,
	\end{align*}
	for $\sigma> 1$. Then, using the  inequality $\zeta(\sigma)\leq \sigma/(\sigma-1)$ and Proposition \ref{Prop8_Bounds_zeta} one arrives at
\begin{mysage}
 saruman = cei(log(2)-RBF(helpconst),3)          
\end{mysage}
	\begin{align}\label{20200303_15:49}
	\begin{split}
	\bigg|\dfrac{1}{\zeta(\sigma+it)}\bigg|&\leq |\zeta(\sigma)|^{3/4}\, |\zeta(\sigma+2it)|^{1/4} \leq \bigg(\dfrac{\sigma}{\sigma-1}\bigg)^{3/4}\big(\log (2t)-\sage{pf(helpconst)}\big)^{1/4}\\
	& \leq \bigg(\dfrac{\sigma}{\sigma-1}\bigg)^{3/4}(\log t+\sage{pf(saruman)})^{1/4}.
	\end{split}
	\end{align}
From \eqref{21_19_6_20} and \eqref{20200303_15:49} we obtain
	\begin{align*}
	\bigg|\dfrac{1}{\zeta(1+it)}\bigg|& \leq \dfrac{e^{\sage{pf(trudbou)}\,d}}{\big|\zeta\big(1+\frac{d}{\log t}+it\big)\big|}\\ &\leq \bigg(1+\dfrac{d}{\log t}\bigg)^{3/4}\bigg(1+\dfrac{\sage{pf(saruman)}}{\log t}\bigg)^{1/4}\dfrac{e^{\sage{pf(trudbou)}\,d}}{d^{3/4}}\log t \\
	& \leq \bigg(1+\dfrac{d}{\log \sage{T1}}\bigg)^{3/4}\bigg(1+\dfrac{\sage{pf(saruman)}}{\log \sage{T1}}\bigg)^{1/4}\dfrac{e^{\sage{pf(trudbou)}\,d}}{d^{3/4}}\log t.
	\end{align*}
        \begin{mysage}
          bestd = 0.0186
          berbro(d) = ((1+d/log(T1))^(3/4))*((1+RBF(saruman)/log(T1))^(1/4))*(exp(RBF(trudbou)*d)/d^(3/4))
          desider = cei(RBF(berbro(RBF(bestd))),3)
          checkme=("MISMATCH!","")[desider<=invzetbound]
        \end{mysage}
        Letting $d=\sage{pf(bestd)}$, we obtain that
        $|1/\zeta(1+it)|\leq \sage{pf(desider)} \log t$ for $t\geq \sage{T1}$.
 \begin{mysage}
def bisecpos(f,a,b):
    """bisecpos(f,a,b): returns whether f(x)>0 for all x in [a,b]. 
    
    Do not call if f crosses or gets very close to the x-axis, as the call may n
ot terminate"""
    y = f(RBF(RIF(a,b)))
    if RBF(y).upper()<=0:
        return False
    if RBF(y).lower()>0:
        return True
    mid = (a+b)/2
    return (bisecpos(f,a,mid) and bisecpos(f,mid,b))

def bisecgr(f1,f2,a,b):
    """bisecgr(f1,f2,a,b): returns whether f1(x)>f2(x) for all x in [a,b]. 
    
    Do not call if f1(x) and f2(x) are ever equal or extremely close to each other, as the call may not terminate"""
    return bisecpos(lambda x: f1(x)-f2(x),a,b)
\end{mysage}
\begin{mysage}
flag=True
for n in [2..T1-1]:
    flag = flag and bisecgr(lambda t:RBF(2.079)*log(t),lambda t:abs(1/zeta(1+i*t)),n,n+1)
checkme=("FALSE!","")[flag]
\end{mysage}
For the case $2\leq t\leq \sage{T1}$, a computation implemented in interval arithmetic
        shows that
	$$
	\bigg|\dfrac{1}{\zeta(1+it)}\bigg|\leq  2.079\log t,
	$$
\sage{checkme} and thus we are done.
\end{proof}

%\section*{\newr{Appendix B. Codes}}
%Here we present some of the codes used in the numerical evaluations.

\subsection*{Acknowledgments}

The authors' work started at the workshop Number Theory in the Americas, which took place at Casa Matem\'atica Oaxaca on August 2019, and was funded
by H.\,A.\,Helfgott's Humboldt professorship (A.\,v.\,Humboldt Foundation) as well as by BIRS/CONACYT.
A.\,Chirre was supported by Grant 275113 of the Research Council of Norway.
E.\,Carneiro was partially supported by Faperj - Brazil. Many thanks are due to C.\,L.\,Aldana, who was a member of our team in Oaxaca
and thereafter. Thanks are also due to D.\,Zagier for
suggesting an improved procedure for
bounding the
infinite product $H(t)$ (see \S \ref{subs:Heffic}), as well as to P.\,Moree, for a related remark, and
to J.\,Maynard for references. Finally, we thank the anonymous referee for the valuable suggestions.

\normalsize

\bibliographystyle{alpha}
\bibliography{oaxaca}

\end{document}